\documentclass{article}%[12pt]{extarticle}
\usepackage[left=3cm,right=3cm,top=3cm]{geometry}
\usepackage{amssymb,amsthm,amsmath,amsxtra,dsfont,appendix,stix}
\usepackage[nottoc]{tocbibind}   % add bibliography to table of contents
\usepackage{hyperref} \hypersetup{colorlinks=true, citecolor=blue, linkcolor=black}
\usepackage{mathrsfs}

\usepackage{tcolorbox}

\newcommand{\Cbeta}{${\rm C\beta E}_N$}

\renewcommand{\i}{\mathbf{i}}
\newcommand{\Hi}{\mathscr{H}}
\newcommand{\C}{\mathds{C}}
\newcommand{\E}{\mathds{E}}
\renewcommand{\P}{\mathds{P}}
\renewcommand{\d}{\mathrm{d}}
\newcommand{\1}{\mathds{1}}
\newcommand{\Co}{\mathscr{C}}
\newcommand{\G}{\mathrm{G}}

\renewcommand{\O}{\mathcal{O}}
\newcommand{\N}{\mathbb{N}}
\newcommand{\R}{\mathds{R}}
\newcommand{\T}{\mathds{T}}
\newcommand{\U}{\mathscr{U}}

\newcommand{\W}{\operatorname{W}}

\newcommand{\sgn}{\operatorname{sgn}}
\newcommand{\Z}{\mathbb{Z}}
\newcommand{\D}{\mathds{D}}

\numberwithin{equation}{section}

\newtheorem{theorem}{Theorem}[section]

\newtheorem{proposition}[theorem]{Proposition}
\newtheorem{corollary}[theorem]{Corollary}
\newtheorem{lemma}[theorem]{Lemma}

\theoremstyle{definition} \newtheorem{remark}{Remark}[section]

\title{
\vspace*{-2cm}
Mesoscopic central limit theorem for the  circular $\beta$ ensembles and applications}
\date{\today}
\author{Gaultier Lambert\thanks{University of Zurich, Switzerland. E-mail: \href{mailto:gaultier.lambert@math.uzh.ch}{gaultier.lambert@math.uzh.ch} }}

\begin{document}

\maketitle

\begin{abstract}{\normalsize
We give a simple proof of a central limit theorem for linear statistics of the Circular $\beta$--ensembles which is valid at almost arbitrary mesoscopic scale and for functions of class $\Co^3$. As a consequence, using a coupling introduced by Valk\`o and Vir\`ag \cite{VV18}, we deduce a central limit theorem for the Sine$_\beta$ processes. 
We also discuss the connection between our result and the theory of Gaussian Multiplicative Chaos.  Based on the result of \cite{LOS18}, we show that the exponential of the logarithm of the real  (and imaginary) part  of the characteristic polynomial of the  Circular $\beta$--ensembles, regularized at a small mesoscopic scale and renormalized, converges to GMC measures in the subcritical regime. This implies that the leading order behavior for the extreme values of the logarithm of the characteristic polynomial is consistent with the predictions of log--correlated Gaussian fields.

% Consequently we obtain leading order estimates for the extreme values of the logarithm of the characteristic polynomial 
%Similar results are also obtained for the fluctuations of a single "eigenvalue" in the spirit of Gustavson~\cite{Gustavson05}.
}\end{abstract}

%{\color{blue}
%In this notes, we investigate the asymptotic behavior of the eigenvalues of the Circular $\beta$-ensembles (C$\beta$E) for general $\beta>0$. The goal is to prove a central limit theorem (CLT) for linear statistics as the dimension $N\to+\infty$ which is valid at all mesoscopic scale and then to deduce a CLT for linear statistics of the Sine$_\beta$ process by using the coupling of \cite{VV}.  As an input for our method, we also need to prove large deviation rigidity estimates for the C$\beta$E eigenvalues configuration. }

%\tableofcontents

\section{Introduction and results}

\subsection{Circular $\beta$--ensembles}  \label{sect:cbe}

%The goal of this note is to study both the fluctuations of linear statistics of the  circular $\beta$--ensembles (C$\beta$E) at (small) mesoscopic scales as well as the fluctuations of individual eigenvalues. 
The circular $\beta$--ensemble or \Cbeta\ for $N\in \N$ is a  point process $0<\theta_1<\cdots <\theta_N <2\pi$ with joint density
\begin{equation} \label{CbN}
\d \P^\beta_{N} = \frac{\Gamma(1+\frac \beta2)^n}{\Gamma(1+n\frac \beta2)} \prod_{1\le j<k \le N} |e^{\i \theta_j}- e^{\i \theta_j}|^\beta \prod_{1\le k \le N} \frac{d\theta_k}{2\pi} , 
\end{equation}
where $\Gamma$ denotes the Gamma function. 
When  $\beta=1,2,4$, these ensembles correspond to the eigenvalues of random matrices sampled according to the Haar measure on the compact groups O$(N)$, U$(N)$ and Sp$(N)$ respectively. 
These ensembles were introduced by Dyson  \cite{Dyson62} as a toy model for scattering matrices or evolution operators coming from quantum mechanics. For general $\beta>0$, \eqref{CbN} corresponds to the  Gibbs measure for $N$ charged particles confined on the circle at temperature $\beta^{-1}$ and interacting via the two--dimensional Coulomb law. For this reason, $\beta$--ensembles are also called log--gases.
It is known that \eqref{CbN} also corresponds to the eigenvalues of certain CMV random matrices \cite{KN04}, so we will refer to the random points $(\theta_j)_{j=1}^N$ as \emph{eigenvalues}.
We refer to Forrester \cite[Chapter 2]{Forrester10} for an in depth introduction to circular $\beta$--ensembles. 

\medskip

We define  the empirical measure by 
$\mu_N = \sum_{j=1}^N \delta_{\theta_j}$ 
and its centered version by 
$\widetilde{\mu}_N = \mu_N - N \frac{d\theta}{2\pi}$. 
 In the following, a linear statistic is a random variable of the form 
\begin{equation} \label{LS}
\int f d \widetilde{\mu}_N =  \sum_{j\le N} f(\theta_j)  - N \widehat{f}_0
\end{equation}
\vspace{-.3cm}

\noindent where $f$ is a continuous function on  $\T= \R/2\pi$ and 
$\displaystyle \widehat{f}_k = \int_\T f(\theta)e^{-\i k\theta} \frac{d\theta}{2\pi}$
for $k\in\Z$ denote the Fourier coefficients of~$f$. 
 Moreover, by \emph{mesoscopic} linear statistic,  we refer to the case where the test function in \eqref{LS} depends on the dimension~$N$ in such a way that 
$f(\theta)= w(L\theta)$ for  $w\in \Co_c(\R)$  and for a sequence   $L=L(N)\to+\infty$ with $L(N)/N \to 0$ as $N\to+\infty$. 
In this regime, it is usual to consider test functions with compact support so that the random variable~\eqref{LS} depends on a large but vanishing fraction of the eigenvalues.

\subsection{Central limit Theorems}

The main goal of this article is to study the fluctuations of linear statistics of the   \Cbeta\ for large $N$ at small mesoscopic scales. 
The circular ensembles are technically easier to analyse than $\beta$--ensembles on $\R$, so this is also an opportunity to give a comprehensive presentation of the method of \emph{loop equation} introduced in \cite{Johansson98}. Then, we discuss applications of our result to the characteristic polynomial of the   \Cbeta\  in section~\ref{sect:charpoly} and we obtain a central limit theorem (CLT) for  the Sine$_\beta$  processes in section~\ref{sect:sine}.

\begin{theorem} \label{clt:meso}
Let $w \in \Co^{3+\alpha}_c(\R)$ for some $\alpha>0$. Let $L(N)>0$ be a sequence such that  $L(N)\to+\infty$ in such a way that  $N^{-1} L(N) (\log N)^3 \to 0$  as $N\to+\infty$ and let 
$w_L(\cdot) = w(\cdot L)$.  Then, we have for any $\beta>0$ as $N\to+\infty$, 
\begin{equation}  \label{Laplace}
\E^\beta_N\left[ \exp\left( \int w_L d \widetilde{\mu}_N  \right)\right]
\to \exp\left( \beta^{-1}\| w\|_{H^{1/2}(\R)}^2 \right).  
\end{equation}
\end{theorem}

The probabilistic interpretation of Theorem~\ref{clt:meso} is that as $N\to+\infty$,
\[
\int w(L\theta) \widetilde{\mu}_N(\d\theta) \to \mathscr{N}\big(0,\tfrac{2}{\beta}\| w\|_{H^{1/2}(\R)}^2\big) 
\]
in law as well as in the sense of moments. The variance of the limiting Gaussian random variable is given by the Sobolev norm:
\begin{equation} \label{var}
\| w\|_{H^{1/2}(\R)}^2 = 2 \int_0^{+\infty} \xi | \widehat{w}(\xi) |^2 d\xi , 
\end{equation}
where $\widehat{w}$ is the Fourier transform of the test function $w$, which is given by
$\displaystyle\widehat{w}(\xi) =\int_\R w(x) e^{-\i x\xi} \frac{dx}{2\pi} $
for $\xi\in\R$.

\medskip

The proof of Theorem~\ref{clt:meso}  is given in section~\ref{sect:CLT}. 
Let us point out that we can actually obtain a  precise control of the error in the asymptotics \eqref{Laplace} 
and that  a straightforward modification of our arguments yields another proof of the CLT for \emph{global} linear statistics.

\begin{theorem} \label{clt:global}
Let $w\in \Co^{3+\alpha}(\T)$ for some $\alpha>0$ be a test function possibly
depending on $N\in\N$ but such that $\|w'\|_{L^1(\R)}$ is fixed. Then there  exists $N_\beta \in \N$ and a $C_{\beta, w}>0$ (which is given e.g. by \eqref{Cw}) such that for all $N\ge N_\beta$, 
\begin{equation} \label{clt}
\left| \log \E^\beta_N\left[ \exp\left( \int w d \widetilde{\mu}_N  \right)\right]
- \beta^{-1} \sigma^2(w)\right| \le C_{\beta,w} \frac{(\log N)^{2}}{ N} , 
\end{equation}
where for any  $f:\T\to\R$ which is sufficiently smooth:  
\begin{equation}  \label{sigma}
\sigma(f)^2 = 2  \sum_{k=1}^{+\infty} k | \widehat{f}_k |^2 ,
\qquad 
\widehat{f}_k = \int_\T f(x) e^{-\i x k} \frac{dx}{2\pi} . 
\end{equation}
\end{theorem}

This CLT for  linear statistics of the \Cbeta\ first appeared in  \cite{Johansson88} for general $\beta>0$ and in the work of Diaconis--Shahshahani \cite{DS94} for the classical values of $\beta=1,2,4$. 
In \cite{Johansson88}, Johansson proved a CLT by using a \emph{transportation method} and, if $\beta=2$, he also discovered a connection between \eqref{clt} and the Strong Szeg\H{o} Theorem, see \cite[Chapter 6]{Simon05}. 
Moreover, because of the rich structure of the circular unitary ensemble (CUE), there exist many other different proofs of the CLT when $\beta=2$, we refer e.g.~to the survey~\cite{DIK13}. 
For general $\beta>0$, there is also a proof of the CLT from Webb \cite{Webb16} based on Stein's method which yields a rate of convergence, but not the precise convergence of the Laplace transform in \eqref{clt}.
Our proof relies on the method of \emph{loop equation} which originates in the work of  Johansson \cite{Johansson98} on the fluctuations of $\beta$-ensembles on $\R$. 
More recently, this method has been refined in \cite{BG13, Shcherbina13,BL18} and it has been applied to two--dimensional Coulomb gases in~\cite{BBNY}. 

\medskip

In the mesoscopic regime, to our knowledge, Theorem~\ref{clt:meso} only appeared for $\beta=2$ in a paper of Soshnikov~\cite{Soshnikov00}. 
Soshnikov's method is very different from ours: it relies on the method of moments and it does not yield the convergence of the Laplace transform of a linear statistics as in Theorem~\ref{clt:meso}. 
For $\beta$-ensembles on $\R$, the mesoscopic CLT was first obtained in \cite{BD16, L18} when $\beta=2$. For general $\beta>0$, it was obtained recently by Bekerman and Lodhia \cite{BL18} using a method of moments based on \emph{higher order loop equations}.

\medskip

It is worth pointing out that our method relies crucially on precise \emph{rigidity estimates} for the eigenvalues.  We obtain such estimates by studying the large deviations for the maximum of the \emph{eigenvalue counting function} in section~\ref{sect:LD}. 
In the next section, we present a consequence for concentration of multi--linear eigenvalues statistics which we believe is of interest.

\subsection{Concentration and eigenvalues rigidity}

 For any function $w\in \Co(\T)$, we define a new biased probability measure:
\begin{equation} \label{biasmeasure}
\d\P^\beta_{N,w} = \frac{ \displaystyle e^{\int w d\mu_N}}{\displaystyle  \E^\beta_N \big[ e^{\int w d\mu_N}\big]} \d\P^\beta_{N} . 
\end{equation}

\begin{proposition} \label{prop:rig}
Let $w\in \Co^1(\T)$ and suppose that $\| w'\|_{L^1(\T)} \le \eta$ where $\eta$ is allowed to depend on $N\in\N$.
There exists $N_\beta \in \N$ such that for all fixed $n\in \N$, all $N\ge N_\beta$ and any $R>0$ (possibly depending on $N\in\N$ as well), we have
\[
\P^\beta_{N,w}\left[ 
\sup_{f\in \mathscr{F}_{n,R}}\bigg| \int_{\T^n} f(x_1,\dots, x_n)  \widetilde{\mu}_N(dx_1)\cdots \widetilde{\mu}_N(dx_n)  \bigg| \ge  R(\tfrac{\sqrt{2}}{\beta} \eta \log N)^{n}  \right] 
\le  \sqrt{\eta \log N} N^{1- \eta^2/2\beta}  , 
\]
where $\mathscr{F}_{n,R} = \bigg\{   f\in \Co^n(\T^n)   :   \displaystyle\int_{\T^n} \left| \frac{\d}{\d x_1} \cdots \frac{\d}{\d x_n} f(x_1,\dots, x_n) \right| dx_1 \cdots dx_n  \le R  \bigg\}$. 
\end{proposition}

The proof of Proposition~\ref{prop:rig} will be given in section~\ref{sect:LD}.
Moreover, we immediately deduce from Lemma~\ref{lem:maxh} below, the following large deviation estimates:  for any $N\in N_\beta$ and  any $R>0$ $($possibly depending on  $N)$, we have
\begin{equation} \label{rig}
\P_N^\beta\left[  \big| \theta_k - \tfrac{2\pi k}{N} \big| \le \frac{2\pi R}{N} ; k=1,\dots , N \right] \ge 1-  3 N e^{- \frac{\beta R^2}{\log N}}.
\end{equation}
This means that the eigenvalues of the \Cbeta\ are close to being equally spaced with overwhelming probability -- this property  is usually called \emph{eigenvalues rigidity}.
In the next section, we obtain optimal rigidity estimates in the sense that we find the leading order of the maximal fluctuations of  $\theta_k$  with the correct constant -- see Corollary~\ref{cor:rig} below. 
These optimal  estimates are obtained through a connection between the characteristic polynomial of the circular $\beta$--ensembles and the theory of Gaussian multiplicative chaos (GMC) that we recall in the next section. 

\subsection{Subcritical Gaussian multiplicative chaos} \label{sect:charpoly}

In this section, we discuss applications of Theorem~\ref{clt:meso} within GMC theory. 
Let $\D=\{z\in\C : |z|<1\}$ and let $P_N$ be the characteristic polynomial of the  \Cbeta, that is for any $N\in\N$, we define
\[
P_N(z) = \textstyle{ \prod_{j=1}^N (1-z e^{-\i\theta_j}) } , \qquad z\in \overline{\D}. 
\]
Our goal is to investigate the asymptotic behavior of $P_N(z)$ for $|z|=1$ as a random function. 
First, let us observe that for any $0<r<1$ and $\vartheta\in\T$, 
\begin{equation} 
\log|P_N(r e^{\i\vartheta})| = - {\textstyle\sum_{j=1}^N}\phi_{r}(\theta_j-\vartheta) ,
\qquad \phi_r(\theta) = \log|1- r e^{\i\theta}|^{-1} . 
\end{equation} 
Hence, $\log|P_N(r e^{\i\vartheta})|$ is a linear statistic and  it follows from Theorem~\ref{clt:global} that 
\begin{equation*}
\log|P_N(z)| \to  \sqrt{\frac{2}{\beta}} \G(z)
\end{equation*}
in the sense of finite dimensional distribution as $N\to+\infty$, where $\G$ is a centered Gaussian process defined on $\D$ with covariance structure:
\begin{equation} \label{Gcov}
 \E\left[\G(z)\G(z') \right]  =   \frac 12 \log\left|1-\overline{z}z' \right|^{-1} , \qquad z,z'\in\D. 
 \end{equation}
We can define the boundary values of the Gaussian process $\G$ as a random generalized function on~ $\T$ and according to formula \eqref{Gcov}, this random field, which is still denoted by  $\G$, is a log--correlated Gaussian process. This process has the same law as $\sqrt{\pi/2}$ times the restriction of the two-dimensional Gaussian free field on $\T$, see \cite[Proposition 1.4]{DRSV17}, so we call it the GFF on $\T$. 
Moreover, one can show that the function $\vartheta \in\T \mapsto \log|P_N(e^{\i\vartheta})|$ converges in law to the random generalized function $\G$ in the Sobolev space $H^{-\delta}(\T)$ for any $\delta>0$, see \cite{HKO01}. 

\medskip

Log--correlated fields form a class of stochastic processes which describe the fluctuations of key observables in many different models related to two--dimensional random geometry, turbulence, finance, etc. 
%%% REF
One of the key universal features of log--correlated fields is their so--called multi-fractal spectrum 
which can be encoded by a family of random measures called GMC measures. 
Within GMC theory, these measures correspond to the exponential of a log--correlated field which is defined by a suitable renormalization procedure. 
For instance, using the results of \cite{RV10} or \cite{Berestycki17}, it is possible to define\footnote{There exist other equivalent ways to define the GMC measures $\mu^\gamma_\G$  that we do not discuss here. We refer to \cite{RV14} for a comprehensive survey of GMC theory.}
\begin{equation} \label{gmc0}
\mu^\gamma_\G(\d\vartheta) = \lim_{r\to1} \frac{e^{\gamma \G(r e^{\i\vartheta})}}{\E e^{\gamma \G(r e^{\i\vartheta})}}  \d\vartheta  . 
\end{equation}
The measure $\mu^\gamma_\G$ exists for all $\gamma\ge 0$, it is continuous in the parameter $\gamma$ and it is non-zero if and only if $\gamma<2$ -- this is called the \emph{subcritical regime}\footnote{Because of the factor $\frac 12$ in formula \eqref{Gcov}, with our conventions,  the critical value is $\gamma_*= 2$.}. 
%We call these random measures the GMC measures associated to the GFF on $\T$. 
The random measure $\mu^\gamma_\G$ is supported on the set of \emph{$\gamma$--thick points}: 
\begin{equation} \label{TP1}
\bigcap_{0\le \alpha < \gamma} \left\{ \theta \in \T : \liminf_{r\to1} \frac{\G(r e^{\i \theta})}{\log|1-r^2|^{-1}} \ge \frac{\alpha}{2}  \right\} . 
\end{equation}
This set is known to have fractal dimension $(1-\gamma^2/4)_+$. 
In particular, if $\gamma_* = 2$ is the critical value,  the fact that the measure $\mu^\gamma_\G$ is non-zero if and only if $\gamma<\gamma_*$ implies that  in probability:
\begin{equation} \label{max0}
\lim_{r\to1} \frac{\max_{\theta \in \T}\G(r e^{\i \theta})}{\frac 12 \log|1-r^2|^{-1}}  = \gamma_*. 
\end{equation}
For a non Gaussian log--correlated field, it is also possible to construct its GMC measures in the subcritical regime. This has been used to describe the asymptotics of powers of the absolute value of the characteristic polynomials of certain ensembles of random matrices, see e.g. Webb and co-authors \cite{Webb15, BWW18} for an application to the circular unitary ensemble $(\beta=2)$, and to a class of Hermitian random matrices, in the so-called $L^2$--regime. Based on the approach from Berestycki \cite{Berestycki17}, a general construction scheme which covers the whole subcritical regime was given in \cite{LOS18} and then refined in our recent work \cite{CFLW}. 
This method has been applied to (unitary invariant) Hermitian random matrices \cite{CFLW}, as well  as to the characteristic polynomial of the Ginibre ensemble \cite{L19}.
A similar approach has also been applied to study the Riemann $\zeta$ function \cite{SW} and cover times of planar Brownian motion~\cite{Jego}.  
Using the method from \cite{LOS18} and relying on the determinantal structure of the circular ensemble  when $\beta=2$ to obtain the necessary asymptotics, Nikula--Saksman--Webb proved in \cite[Theorem~1.1]{NSW} that  for any $0\le \gamma <2$, 
\begin{equation} \label{cuegmc}
\frac{|P_N(e^{\i\vartheta})|^\gamma}{\E_N^2\left[ |P_N(e^{\i\vartheta})|^\gamma\right]} \d\vartheta \to \mu^\gamma_\G(\d\vartheta)
\end{equation}
in distribution as $N\to+\infty$.
It is a very interesting and challenging problem to generalize \eqref{cuegmc} to all $\beta>0$.
In the following, we provide the first step in this direction which consists in  constructing the GMC measures associated with a small mesoscopic regularization of the characteristic polynomial $P_N$.  
Namely, by adapting the proof of Theorem~\ref{clt:meso}, we are able to obtain the following result:

\begin{theorem} \label{thm:gmc1}
Let $r_N = 1- \frac{(\log N)^6}{N}$ and, by analogy with \eqref{gmc0},  define the random measure for any $\gamma \in\R$, 
\begin{equation} \label{gmc1}
\mu_N^\gamma(d \theta) = \frac{|P_N(r_N e^{\i\theta})|^\gamma}{\E_N^\beta\left[ |P_N(r_N e^{\i\theta})|^\gamma \right]} \frac{d\theta}{2\pi} .
\end{equation}
For any $|\gamma| \le \sqrt{2\beta}$ (i.e.~in the subcritical regime), $\mu_N^\gamma$ converges in law as $N\to+\infty$ to a GMC measure $\mu_\G^{\breve\gamma}$ associated to the GFF on $\T$ with parameter $\breve{\gamma} = \gamma  \sqrt{\frac{2}{\beta}}$. 
\end{theorem}

We can also obtain an analogous result for the imaginary part of the logarithm of the characteristic polynomial of the \Cbeta. Let
\begin{equation} \label{imcharpoly}
\Psi_{N,r}(\vartheta) =  \textstyle{ \sum_{j=1}^N } \Im\log(1-r e^{\i(\vartheta-\theta_j)}) , \qquad  r \in [0,1), \quad \vartheta\in\T, 
\end{equation}
where $\log(\cdot)$ denotes the principle branch\footnote{This is the usual convention used  e.g. in \cite{HKO01, ABB17, CMN18, NSW}.} of the logarithm so that the function $\Im\log(1-z)$ is analytic for $z\in\D$. We also let
\begin{equation} \label{Psi}
\Psi_N(\vartheta) = \lim_{r\to 1} \Psi_{N,r}(\vartheta)  =  \textstyle{ \sum_{j=1}^N } \psi(\vartheta - \theta_j) 
, \qquad  \vartheta\in\T, 
\end{equation}
where  $\psi(\theta) =  \Im \log(1-e^{\i\theta}) = \frac{\theta-\pi}{2} $ for all $\theta \in (0, 2\pi)$.

\begin{theorem} \label{thm:gmc2}
Let $r_N = 1- \frac{(\log N)^4}{N}$ and define the random measure for any $\gamma \in\R$, 
\begin{equation} \label{gmc2}
\widetilde{\mu}_N^\gamma(d \theta) = \frac{e^{\gamma \Psi_{N,r_N}(\theta)}}{\E_N^\beta\big[e^{\gamma \Psi_{N,r_N}(\theta)}\big]} \frac{d\theta}{2\pi} .
\end{equation}
For any $|\gamma| \le \sqrt{2\beta}$, $\widetilde{\mu}_N^\gamma$ converges in law as $N\to+\infty$ to a GMC measure $\mu_\G^{\breve\gamma}$  with parameter $\breve{\gamma} = \gamma  \sqrt{\frac{2}{\beta}}$. 
\end{theorem}

It is known that the supports of the random measures $\mu_N^\gamma$ correspond to the \emph{thick points} of the characteristic polynomial $P_N$, see e.g. \cite[section~3]{CFLW}. 
By analogy with \eqref{TP1}, these  \emph{thick points} are the atypical points where $|P_N|$ takes extremely large values. 
Concretely, for any $\gamma>0$, we say that $\theta \in \T$ is a $\gamma$--thick if the value of $\log  | P_N(e^{\i\theta})|$ is at least
$ \gamma \E^\beta_N \left[(\log  | P_N(e^{\i\theta})|)^2 \right] = \frac{\gamma}{\beta} \log N +\O(1)$. 
In section~\ref{sect:TP}, we show how to deduce from Theorem~\ref{thm:gmc1} and the asymptotics from \cite[Theorem 1.2]{Su09}  that the size of the sets of   \emph{thick points} are given  according to  the predictions of log--correlated Gaussian fields.

\begin{proposition} \label{prop:TP1}
For any $\gamma>0$, let 
\begin{equation} \label{TP2}
\mathscr{T}^\gamma_N = \left\{ \theta \in \T :  | P_N(e^{\i\theta})|   \ge N^{\gamma/\beta} \right\} 
\end{equation}
and $|\mathscr{T}^\gamma_N|$ be the Lebesgue measure of the set $\mathscr{T}^\gamma_N$. 
Then for any $\gamma < \gamma_* = \sqrt{2\beta}$,  we have 
$\displaystyle\frac{\log|\mathscr{T}^\gamma_N |}{ \log N} \to - \frac{\gamma^2}{2\beta}$
 in probability as $N\to+\infty$. 
Moreover, we have in probability  as $N\to+\infty$,
\begin{equation} \label{max1}
\frac{\max_{\theta\in\T} \log| P_N(e^{\i \theta}) | }{\log N} \to \frac{\gamma_*}{\beta} = \sqrt{\frac{2}{\beta}}  .
\end{equation}
\end{proposition}

%Let us point out that compared with \cite{ABB17}, these estimates hold for all $\beta>0$ and 
The interpretation of Proposition~\ref{prop:TP1} is that the \emph{multi--fractal spectrum} of the sets of $\gamma$-thick points of the \Cbeta\ characteristic polynomial is given by the function $\gamma \mapsto (1- \gamma^2/2\beta)_+$ for $\gamma \ge 0$. This is in accordance with the behavior of Gaussian log--correlated fields.  
Proposition~\ref{prop:TP1} was first obtained  by Arguin--Belius--Bourgade \cite[Theorem 1.3]{ABB17} for the CUE $(\beta=2)$. We generalize this result for all $\beta>0$. 
Then, by \cite[Corollary 1.4]{ABB17}, we also obtain the limit of the so-called \emph{free energy}:
\begin{equation*} \label{freezing}
\lim_{N \to+\infty} \frac{1}{\log N} \log\left( \int_\T | P_N(e^{\i \theta}) |^\gamma \frac{d\theta}{2\pi} \right)
= \begin{cases}
\frac{\gamma^2}{2\beta} &\text{if } \gamma \in [0,\gamma_*] \\
\sqrt{\frac{2\gamma^2}{\beta}} -1  &\text{if } \gamma > \gamma_* 
\end{cases} . 
\end{equation*}
This shows an interesting \emph{transition} at the critical value $\gamma_* = \sqrt{2\beta}$. For log--correlated fields, the fact that the free energy becomes linear in the \emph{super--critical regime} $(\gamma>\gamma_*)$ is usually called \emph{freezing}. 
In particular, this \emph{freezing phenomenon} plays a crucial role in predicting the precise asymptotic behavior of $|P_N|$, see  Fyodorov--Keating \cite{FK14}. 
We can also obtain a result analogous to Proposition~\ref{prop:TP1} for the imaginary part of the logarithm of the characteristic polynomial \eqref{Psi}. 

\begin{proposition} \label{prop:TP2}
For any $\gamma>0$, let 
\[
\widetilde{\mathscr{T}}^{\gamma}_N = \left\{ \theta \in \T :  \Psi_N(\theta)  \ge \frac \gamma\beta \log N \right\} .
\]
Then for any $\gamma < \gamma_*$,   we have  $\frac{\log|\widetilde{\mathscr{T}}^\gamma_N |}{ \log N} \to -\frac{\gamma^2}{2\beta}$ in probability as $N\to+\infty$.
Moreover, we have in probability  as $N\to+\infty$,
\begin{equation} \label{max2}
\frac{\max_\T \Psi_N(\theta)}{\log N} \to \frac{\gamma_*}{\beta} = \sqrt{\frac{2}{\beta}} . 
\end{equation} 
\end{proposition}

\begin{remark} \label{rk:symmetry}
Since the function $-\psi(\theta) = \psi(2\pi-\theta)$ for all $\theta\in(0,2\pi)$, we see that as random field: $\Psi_N(\theta) \overset{\rm law}{=} -\Psi_N(-\theta)$. 
By~\eqref{max2}, this implies for instance that $\displaystyle\frac{\min_\T \Psi_N(\theta)}{\log N} \to \sqrt{\frac{2}{\beta}}$ in probability as $N\to+\infty$. 
\hfill $\blacksquare$\end{remark}

The law of large numbers \eqref{max1} and \eqref{max2} for the maximums of the real and imaginary parts of the logarithm of the characteristic polynomial of the \Cbeta have already been obtained in \cite{CMN18}  by a completely different method. 
In fact, the complete asymptotic behavior of the maximum of the field $\log|P_N|$ when $\beta=2$ was predicted in \cite{FK14} by analogy with Gaussian log--correlated fields and part of this conjecture was verified by Chhaibi--Madaule--Najnudel \cite[Theorem~1.2]{CMN18} who showed that $\max_\T  \log| P_N|$ and $\max_\T \Psi_N$, once re-centered, are tight random variables. 
Let us also point out that extensive numerical studies of the extreme value statistics of the \Cbeta\ characteristic polynomial for large $N\in\N$  have been done by 
Fyodorov--Gnutzmann--Keating \cite{FGK18} and they indicate some interesting relationships between the extreme values of the logarithm of the characteristic polynomial and large gaps in the spectrum. 

\medskip

Finally, as observed in \cite[Theorem~1.5]{ABB17} or  \cite[Corollary~1.3]{CMN18},  the asymptotics \eqref{max2} imply \emph{optimal rigidity estimates} for the \Cbeta\ eigenvalues.

\begin{corollary} \label{cor:rig}
For any $\beta >0$ and $\delta>0$, 
\begin{equation} \label{oprig}
\lim_{N\to+\infty} \P_N^\beta\left[ (2-\delta) \sqrt{\frac{2}{\beta}}  \frac{\log N}{N}  \le \max_{k=1,\dots , N}  \big| \theta_k - \tfrac{2\pi k}{N} \big| \le (2+\delta) \sqrt{\frac{2}{\beta}}  \frac{\log N}{N} \right] =1 . 
\end{equation}
\end{corollary}

\subsection{Sine$_\beta$ point processes}  \label{sect:sine}

The Sine$_\beta$ processes describe the bulk scaling limits of the eigenvalues of  $\beta$-ensembles. 
This family of translation invariant point processes on $\R$ was first introduced independently by Killip--Stoiciu \cite{KS09} as the scaling limits of the \Cbeta\ and by Valk\'o--Vir\'ag \cite{VV09} as that of Gaussian $\beta$-ensembles. 
For general $\beta>0$, universality of the  Sine$_\beta$ processes in the bulk $\beta$--ensembles on $\R$ was obtained by Bourgade--Erd\H{o}s--Yau  \cite{BEY14}  for the class of analytic one--cut regular potential by  coupling two different ensembles  using the Dyson Brownian motion. 
Our proof of Theorem~\ref{clt:sine} relies on a similar idea. Using the coupling  from Valk\'o--Vir\'ag \cite{VV18} between the  Sine$_\beta$  and \Cbeta\ point processes, we can transfer our mesoscopic CLT (Theorem~\ref{clt:meso}) into a CLT for the Sine$_\beta$ processes.

\medskip

The  Sine$_\beta$ process is usually defined through its counting function which satisfies a system of stochastic differential equations \cite{KS09, VV09}. 
Recently, Valk\'o--Vir\'ag \cite{VV17}  introduced an alternate characterization as the eigenvalues of a stochastic differential operator.
It turns out that the \Cbeta\ also corresponds to the eigenvalues  of an operator of the same kind as the  Sine$_\beta$ and it is possible to couple these two operators  in such a way that their eigenvalues are close to each other. 
This coupling was studied in detail by Valk\'o--Vir\'ag  \cite{VV18} and they obtain the following result. 

\begin{theorem} \label{thm:coupling}
Fix $\beta>0$ and recall that $0<\theta_1<\cdots <\theta_N <2\pi$ denotes the eigenvalues of  \Cbeta. Let us extend this configuration periodically by setting   $\theta_{k+ \ell N} = \theta_k + 2\pi \ell$ for all $k\in [N]$ and $\ell\in\Z$.
%Let $\P^\beta_{\rm Sine}$  be the law of the Sine$_\beta$ point process and let $(\lambda_k)_{k\in\Z}$ be a configuration of the Sine$_\beta$ process labelled in such a way that $0<\lambda_1 <\lambda_2 <\cdots $ and $0<-\lambda_0 < - \lambda_{-1} < \cdots$   $\P^\beta_{\rm Sine}$ almost surely. 
By $\cite[Corollary~2]{VV18}$, there exists a coupling $\P$ of the \Cbeta\  with the   Sine$_\beta$ process $(\lambda_k)_{k\in\Z}$ such that for any $\epsilon>0$, there exists a random integer $\mathrm{N}_\epsilon$ and we have for all $N\ge \mathrm{N}_\epsilon$, 
\[
 \big| \tfrac{N}{2\pi}\theta_k - \lambda_k \big| \le \frac{1+k^2}{N^{1/2-\epsilon}}  , \qquad \forall |k| \le N^{1/2-\epsilon} . 
\]
\end{theorem}

As a consequence of the coupling of Theorem~\ref{thm:coupling}  from  \cite{VV18} and Theorem~\ref{clt:meso}, we easily obtain the following result.
The details of the proof will be given in section~\ref{sect:VV}.

\begin{theorem} \label{clt:sine}
Let $(\lambda_k)_{k\in\Z}$ be a configuration of the 
Sine$_\beta$ process and let $w\in \Co^{3+\alpha}_c(\R)$  for some $\alpha>0$. %For any $\nu>0$, let $w_{1/\nu}(\cdot) = w(\cdot/\nu)$. 
We have as $\nu \to+\infty$
\begin{equation*} 
\sum_{k\in \Z} w(\lambda_k \nu^{-1}) -  \nu \int_\R w dx\ \to\  \mathscr{N}\big(0,\tfrac{2}{\beta}\| w\|_{H^{1/2}(\R)}^2\big).
\end{equation*}
The convergence holds  in distribution and the limiting variance \eqref{var} is the same as in Theorem~\ref{clt:meso}. 
\end{theorem}

Let us mention that for $\beta=2$, other couplings  between the CUE and 
Sine$_2$ existed prior to \cite{VV17,VV18}.  
For instance, the work of Maple--Najnudel--Nikeghbali \cite{MNN18} based on \emph{virtual isometries}  and the work of Meckes--Meckes \cite{MM16} which uses the determinantal structure of these processes. 
Moreover, it is possible to obtain Theorem~\ref{clt:sine} directly  by using the determinantal structure of the Sine$_2$ process,  see Kac \cite{Kac54} and Soshnikov \cite{Soshnikov00}. 

\medskip

Finally, it should be mentioned that there have been plenty of recent developments in the study of the Sine$_\beta$ for general $\beta>0$.  
Using the SDE representation, large deviation estimates  for the number of eigenvalues in a box were obtained in \cite{VV10, HV15, HV17}, as well as a CLT  in \cite[Theorem 17]{KVV12}. The rigidity property for Sine$_\beta$ in the sense of Gosh--Peres was proved by Chhaibi--Najnudel \cite{CN18} and Holcomb--Paquette  \cite{HP18} computed the leading order of the maximum eigenvalues counting function. 
Finally, Lebl\'e \cite{Leble18} gave  recently an alternate proof of Theorem~\ref{clt:sine} for test functions of class $\Co^{4}_c(\R)$ which relies on the \emph{DLR equations} for the Sine$_\beta$ process established by Dereudre-Hardy-Lebl\'e-Ma\"ida \cite{DHLM}.

\subsection{Organization of the paper}

In section~\ref{sect:CLT}, we prove our main results Theorems \ref{clt:meso} and \ref{clt:global} by using the method of \emph{loop equation} which we review in section~\ref{sect:leq}.
In section~\ref{sect:chaos}, we discuss applications from the perspective of Gaussian multiplicative chaos.  
Specifically, in sections~\ref{sect:gmc1} and~\ref{sect:gmc2}, we explain how to modify the proof of  Theorem~\ref{clt:meso} in order to obtain Theorem~\ref{thm:gmc1} and Theorem~\ref{thm:gmc2} respectively. Then, we give the proofs of Propositions~\ref{prop:TP1} and~\ref{prop:TP2} in section~\ref{sect:TP}.
In section~\ref{sect:LD}, we obtain rigidity results for the circular $\beta$--ensemble  by studying the large deviations of the eigenvalue counting function.
In particular, we prove Proposition~\ref{prop:rig} which is a key input in our proof of Theorem \ref{clt:meso}. 
Finally, in section~\ref{sect:VV}, we give the short proof of Theorem~\ref{clt:sine}. 
\subsection{Acknowledgment}

G.L. is  supported by the University of Zurich Forschungskredit grant FK-17-112 and  by the grant SNSF Ambizione  S-71114-05-01.
G.L. wishes to thank E. Paquette and C. Webb for interesting discussions about the problems studied in this article.

\section{Proof of Theorem~\ref{clt:meso}} \label{sect:CLT}

\subsection{Loop equation} \label{sect:leq}

\begin{lemma}[Loop equation] \label{lem:Leq}
Let $w\in C^1(\T)$ and  $\P^\beta_{N, w}$ be as in \eqref{biasmeasure}. 
Recall that we let $\mu_N = \sum_{j=1}^N \delta_{\theta_j}$.
For any $g\in C^1(\T)$ and any $N\in\N$,  we have
\[
\E^\beta_{N, w} \left[ \frac{\beta}{2} \iint \frac{g(x)-g(t)}{2 \tan(\frac{x-t}{2})} \mu_N(dx) \mu_N(dt)
+ (1-\frac \beta 2) \int g' d\mu_N + \int gw' d\mu_N  \right] =0 . 
\]
\end{lemma}

The proof of Lemma~\ref{lem:Leq} is straightforward, it relies on the definition of the biased measure $\P^\beta_{N,w}$, the explicite density \eqref{CbN} and an integration by parts -- we refer to \cite[formula (2.18)]{Johansson98} for the  analogous formula  for $\beta$--ensembles on~$\R$. 
In order to obtain Theorem~\ref{clt:meso} from Lemma~\ref{lem:Leq},  we need to let\footnote{ $g= \U w$ is also called the harmonic conjugate of $w$. To understand why this choice is relevant, we refer to formula \eqref{W5} below.} $g= \U w$, the \emph{Hilbert transform} of the function $w$. 
The Hilbert transform  $\U$ on $L^2(\T)$  is a bounded operator defined in such a way that for any $k\in\Z$, 
\begin{equation} \label{UT}
\U(e^{\i k\theta}) =- \i \sgn(k)  e^{\i k\theta}  ,
\end{equation}
where $\sgn(k)= \pm$ if $k\in \Z_{\pm}$ and $\sgn(0) =0$.
This operator is invertible on $L^2_0(\T)$ with $\U^{-1} = - \U$ and it has the following integral representation: for any $ f \in \Co^\alpha(\T)$ with $\alpha>0$,
\begin{equation} \label{U}
\U f (x) =- \int_\T  \frac{f(x)-f(t)}{\tan(\frac{x-t}{2})} \frac{dt}{2\pi} ,
\qquad x\in\T.  
\end{equation}
Further properties of the Hilbert transform that we shall use in the proofs are recorded by the next Proposition.

\begin{proposition}\label{prop:U}
We have $(\widehat{\U f})_0 = 0$ for any $f\in L^2(\T)$. 
Moreover, if $f$ is differentiable with $f' \in L^2(\T)$, then 
$(\U f)' = \U(f')$ and $\| (\U f)' \|_{L^2(\T)} = \|f'\|_{L^2(\T)}$. 
In particular, this implies that the function $\U f$ is absolutely continuous on $\T$ and $\| \U f\| \le \sqrt{2\pi} \|f'\|_{L^2(\T)}$. 
\end{proposition}

These basic properties are easy to verify, so we skip the proof of Proposition~\ref{prop:U}.
Our CLT follows from the following lemma and  technical estimates on the random variable \eqref{W4} that we discuss in sections~\ref{sect:globalclt} and~\ref{sect:mesoclt}.

\begin{lemma}\label{lem:Jo}
Let $\widetilde{\mu}_N$ be as in \eqref{LS}.
Let  $w\in \Co^{2}(\T)$ be a function which may depend on $N\in\N$, let  $g=\U w$  and define for any $t>0$, 
\begin{equation} \label{W4}
\widetilde{\W}_N(w) =
\frac{\beta}{2} \iint \frac{g(x)-g(t)}{2 \tan(\frac{x-t}{2})} \widetilde{\mu}_N(dx) \widetilde{\mu}_N(dt) + (1-\frac \beta 2) \int g' d\widetilde{\mu}_N + t \int gw' d\widetilde{\mu}_N . 
\end{equation}
If $ \displaystyle  \delta_N(w)  =  \frac{2}{\beta N}  \sup_{t\in [0,1]} \left|  \E^\beta_{N, tw}[ \widetilde{\W}_N ] \right|$ and $\sigma^2(w)$ is given by formula \eqref{sigma}, then 
\vspace*{-.3cm}
\begin{equation*} 
\left| \log \E^\beta_N\left[ e^{ \int w d\widetilde{\mu}_N} \right] - \beta^{-1}\sigma^2(w) \right| \le \delta_N .
\end{equation*}
\end{lemma}

\begin{proof}
The result of Lemma~\ref{lem:Jo} is  classical, we give a quick proof for completeness. 
Let 
\begin{equation} \label{W1}
\W_N = \frac{\beta}{2} \iint \frac{g(x)-g(u)}{2 \tan(\frac{x-u}{2})} \mu_N(dx) \mu_N(du) + (1-\frac \beta 2) \int g' d\mu_N +  t \int gw' d\mu_N . 
\end{equation}
First of all, we observe that replacing
$ \mu_N(dx) =  \widetilde{\mu}_N(dx) + \frac{dx}{2\pi}$ in formula \eqref{W1}, 
by \eqref{U} and since $(\widehat{g'})_0 =0$, we obtain
\begin{equation} \label{W5}
\W_N =  \frac{N\beta}{2} \int  \U g  d\widetilde{\mu}_N -   \frac{N^2\beta}{4} \int  \U g(x) \frac{dx}{2\pi} + t N  \int g(x)w'(x)  \frac{dx}{2\pi} 
+   \widetilde{\W}_N ,
\end{equation}
where $\widetilde{\W}_N$ is given by \eqref{W4}.
Since $(\widehat{\U g})_0 = 0$  and $\E^\beta_{N, t w} \left[  \W_N \right] =  0$ for any $t>0$  by Lemma~\ref{lem:Leq},  this implies that 
\begin{equation} \label{W2}
-  \E^\beta_{N, w}\left[ \int \U g d\widetilde{\mu}_N \right]   =  \frac{2t}{\beta}  \int g(x)w'(x)  \frac{dx}{2\pi}+  \frac{2}{\beta N} \E^\beta_{N, w}[\widetilde{\W}_N ] . 
\end{equation}

Now, by Parseval's theorem and \eqref{UT}, observe that  according to formula \eqref{sigma}, we have
\begin{equation} \label{sigmaU}
 \int_\T g(x)w'(x)  \frac{dx}{2\pi} =  \sum_{k\in\Z} (-\i k) \widehat{w}_k \overline{\widehat{g}_k} =  \sum_{k\in\Z}  |k| |  \widehat{w}_k |^2 = \sigma(w)^2 .
 \end{equation}
Since $\U g = -w $ by definition of the Hilbert transform, by \eqref{W2}--\eqref{sigmaU}, we obtain 
\begin{equation} \label{le3}
\left|  \E^\beta_{N, tw}\left[ \int w d\widetilde{\mu}_N \right]   -  \frac{2t}{\beta} \sigma^2(w)  \right| \le \delta_N . 
\end{equation}
Now, by \cite[formula (2.16)]{Johansson98},  observe that for any $t\in (0,1]$,
\begin{equation*} 
\frac{\d}{\d t} \log \E^\beta_N\left[ e^{ t\int w d\widetilde{\mu}_N} \right]
= \E^\beta_{N, tw}\left[ \int w d\widetilde{\mu}_N \right] .
\end{equation*}
So, if we integrate the LHS of \eqref{le3} with respect to $t\in (0,1]$, this completes the proof.
\end{proof}

Hence, in order to prove Theorems~\ref{clt:meso} and~\ref{clt:global}, we have to estimate the error term $\delta_N$ from Lemma~\ref{lem:Jo}  in the mesoscopic, respectively global, regimes. 
This will be done carefully in the next two sections. 

\subsection{Estimates in the global regime: Proof of Theorem~\ref{clt:global}}
\label{sect:globalclt}

In this section, we use our rigidity estimates form Proposition~\ref{prop:rig} to estimate the error term in Lemma~\ref{lem:Jo}. 

\begin{proposition} \label{prop:W}
Let $w\in \Co^{3+\alpha}(\T)$ for some $\alpha>0$ be a function which may depend on $N\in\N$ in such a way that  $\| w'\|_{L^1(\T)} \le c$ for some fixed $c\ge 1$ and  let  $g=\U w$. 
Let $N_\beta \in \N$ be as in Proposition~\ref{prop:rig} and  let $\delta_N(w)$ be as in Lemma~\ref{lem:Jo}.  There exists  a constant $C_\beta>0$ which only depends on $\beta>0$ and $c>0$ such that all $N\ge N_\beta$ and $t\in[0,1]$, 
\begin{equation} \label{W3}
\delta_N(w)
\le C_\beta \left(   R_0 \log N+ R_1 +  R_2  N^{-5} \right)\frac{\log N}{N}  ,
\end{equation}
where
\begin{equation} \label{R12}
R_1(w) =   \| g''\|_{L^1(\T)} + \| (gw')'\|_{L^1(\T)} , \qquad 
R_2(w) =  \| g'\|_{\infty} + \| g\|_\infty \| w'\|_\infty
\end{equation}
and 
\begin{equation} \label{R0}
R_0(w)=    \iint_{\T^2} \bigg|\frac{ g(x_1)-g(x_2) -  (g'(x_1) + g'(x_2)  \tan(\frac{x_1-x_2}{2})}{4 \sin^2(\frac{x_1-x_2}{2}) \tan(\frac{x_1-x_2}{2})}\bigg|\d x_1 \d x_2  . 
\end{equation}
\end{proposition}

\begin{proof}
Since 
$ \displaystyle \Big| \int f d\widetilde{\mu}_N  \Big|  \le 2\| f\|_\infty N$
for any $f\in \Co(\T)$,  by Proposition~\ref{prop:rig} applied with $\eta \ge c$,  we obtain for all $N\ge N_\beta$, 
 \begin{align*} \notag
\E^\beta_{N, tw} \left[  \Big| \int g' d\widetilde{\mu}_N  \Big| \right] 
& \le 2\| g'\|_\infty N \P^\beta_{N, tw} \left[  \Big| \int g' d\widetilde{\mu}_N  \Big| \ge  \tfrac{\sqrt{2}}{\beta} \eta \| g''\|_{L^1(\T)}  \log N\right] + \tfrac{\sqrt{2}}{\beta} \eta  \| g''\|_{L^1(\T)}  \log N \\
& \le   2\| g'\|_\infty  \sqrt{\eta \log N} N^{2- \eta^2/2\beta} + \tfrac{\sqrt{2}}{\beta} \eta  \| g''\|_{L^1(\T)}  \log N
%C_\beta \left( \| g'\|_\infty    N^{-6} + \| g''\|_{L^1(\T)} \right)  \log N   ,
\end{align*}  
Similarly, we have
\begin{equation*}
\E^\beta_{N, w} \left[  \Big| \int gw' d\widetilde{\mu}_N  \Big| \right]  \le 
2 \| g\|_\infty \| w'\|_\infty \sqrt{\eta \log N} N^{2- \eta^2/2\beta}+ \tfrac{\sqrt{2}}{\beta} \eta  \| (gw')'\|_{L^1(\T)} \log N
\end{equation*}
and
\[
\E^\beta_{N, w} \left[  \Big| \iint \frac{g(x)-g(u)}{2 \tan(\frac{x-u}{2})} \widetilde{\mu}_N(dx) \widetilde{\mu}_N(du)  \Big| \right]  \le  
4  \| g'\|_\infty\sqrt{\eta \log N} N^{3- \eta^2/2\beta}
+  \frac{2}{\beta^2} \eta^2 R_0 ,
% C_\beta \left(  \| g'\|_\infty N^{-5}+ R_0 \log N \right)  \log N
\]
where we used that $\sup_{x, u \in\T}  \Big| \frac{g(x)-g(u)}{2 \tan(\frac{x-u}{2})}  \Big| \le \| g'\|_\infty $ in Proposition~\ref{prop:rig} with $n=2$   and
\[
R_0 =   \iint_{\T^2} \bigg| \frac{\d}{\d x_1}\frac{\d}{\d x_2} \frac{ g(x_1)-g(x_2)}{2 \tan(\frac{x_1-x_2}{2})}\bigg|\d x_1 \d x_2  .
\]
By an explicit computation, we verify that $R_0$ is given by \eqref{R0}. According to \eqref{W4}, using the triangle inequality and collecting all the terms, we obtain that there exists a universal constant $C>0$ such that for all $N\ge N_\beta$ and $t\in[0,1]$, 
\begin{equation} \label{W6}
\E^\beta_{N, t w} \left[ | \widetilde{\W}_N | \right] 
 \le C \sqrt{\eta\log N} R_2(w) N^{3-\eta^2/2\beta} + \tfrac{2}{\beta^2}R_0(w) (\eta \log N)^2 +\tfrac{\sqrt{2}}{\beta} \eta R_1(w) \log N . 
\end{equation}
Taking $\eta= c+ 4\sqrt{\beta}$, since $ \displaystyle  \delta_N(w)  =  \frac{2}{\beta N}  \sup_{t\in [0,1]} \left|  \E^\beta_{N, tw}[ \widetilde{\W}_N ] \right|$,  we obtain the inequality \eqref{W3}. 
\end{proof}

We are now ready to give the proof of Theorem~\ref{clt:global}. 

\begin{proof}[Proof of Theorem~\ref{clt:global}] 
%With the notation of Lemma~\ref{lem:Jo} and Proposition~\ref{prop:W}. 
Since we assume that $w\in \Co^{3+\alpha}(\T)$, by Proposition~\ref{prop:U}, we have $g\in \Co^{3}(\T)$ and  the terms \eqref{R12} satisfy\footnote{This is a straightforward computation using that $\|f\|_{L^1(\T)} \le \sqrt{2\pi} \|f\|_{L^2(\T)} \le  2\pi \| f\|_\infty$.}
\[
R_1(w) , R_2(w) \le C (1 +\|w'\|_\infty^2 + \| w''\|_\infty + \|w'''\|_\infty )
\]
for some universal constant $C>0$. 
In order to estimate $R_0$, observe that by Taylor's theorem, the integrand in \eqref{R0} is uniformly bounded by $\|g'''\|_{\infty}$ so that
$R_0(w) \le C \|g'''\|_{\infty}$.
Combining these estimates with Lemma~\ref{lem:Jo},  
we obtain \eqref{clt:global} with 
\begin{equation} \label{Cw}
C_{\beta, w} = C_\beta(1+ \|w'\|_\infty^2 +  \|w'''\|_\infty +  \|\U w '''\|_{\infty} ) 
\end{equation} 
and $C_\beta =\frac{C}{\beta^{3/2}} (1+\frac{1}{\beta^{3/2}})$ for some universal constant $C>0$. 
This completes the proof.
\end{proof}

\subsection{Estimates in the mesoscopic regime: Proof of Theorem~\ref{clt:meso}} \label{sect:mesoclt}

In comparison to the argument given in the previous section, to obtain our mesoscopic CLT at small scales,  we need more precise estimates for the error $\delta_N(w_L)$, see \eqref{W3}, especially for the term $R_0$ \eqref{R0}.

\medskip

In this section, we fix $ w\in \Co_c^{3+\alpha}(\R)$ for some $\alpha>0$. 
Without loss of generality, we assume that $\operatorname{supp}(w) \subseteq [-\frac \pi2, \frac \pi2]$.
For any $L\ge 1$, we let $w_L(\cdot) = w(\cdot L)$.
We may treat $w_L$ has a $2\pi$--periodic function in $\Co^{3+\alpha}(\T)$ and  set $g_L = \U w_L$ where $\U$ is the Hilbert transform \eqref{U}. 
In particular,  $g_L \in  \Co^{3}(\T)$ by Proposition~\ref{prop:U}.

\medskip

For any $f\in\Co^\alpha(\R)$ for some $\alpha>0$ with  $\operatorname{supp}(f) \subseteq [-\frac \pi2, \frac \pi2]$,
 we define
\begin{equation} \label{UL}
\U_L f (x) = - \int_{-\pi L}^{\pi L} \frac{f(x) - f(t)}{2\pi L \tan\left(\frac{x-t}{2L}\right)}dt  ,
\qquad x\in\R .
\end{equation}
The following proposition is useful for our proof.

\begin{proposition} \label{prop:gL}
With the above convention, for any $k=0,1,2,3$ and for all $x\in [-\pi , \pi ]$,
\begin{equation} \label{gL}
 g_L^{(k)}  (x)  =   L^k \U_L(w^{(k)}) (xL)  ,
\end{equation}
where $g_L^{(k)}$, $w^{(k)}$ denotes the derivatives of the functions $g_L \in \Co^3(\T)$ and $w \in \Co^{3+\alpha}(\R)$ respectively.  
Moreover, we have
\begin{equation} \label{gLest}
\| g_L^{(k)} \|_{L^1(\T)}  \le r_{k ,w}  L^{k-1} \log(\pi L)
\end{equation}
where $r_{k,w} =  2 \| w^{(k)}\|_{L^1(\R)} + 2\pi c  \| w^{(k)}\|_\infty +   2\pi c_\alpha \| w^{(k)}\|_{\Co^\alpha}$ for $k=0,1,2,3$ and universal constants $c,c_\alpha>0$. 
\end{proposition}

\begin{proof}
First of all, observe that by a change of variable, for any $x\in[-\pi, \pi]$, 
\[ \begin{aligned}
g_L(x) 
 & = \U w_L(x)= - \int_{-\pi L}^{\pi L} \frac{w(xL)-w(t)}{\tan\left(\frac{xL-t}{2L}\right)} \frac{dt}{2\pi L} \\
&= \U_L w(x L) . 
\end{aligned}\]
This establishes formula \eqref{gL} for $k=0$ -- the other cases follow in a similar way by observing that according to Proposition~\ref{prop:U}, the function $g_L \in \Co^3(\T)$ and $g_L^{(k)}= \U (w_L^{(k)}) $ for $k=1,2,3$. 
In order to obtain the estimate \eqref{gLest}, we use that for any $0<\alpha \le 1$, there exist universal constants $c,c_\alpha>0$ such that for any function $f\in\Co^\alpha(\R)$ with $\operatorname{supp}(f) \subseteq [-\frac \pi2, \frac \pi2]$, 
\begin{equation} \label{Uest}
\left| \U_L f (x)  \right| \le 
\begin{cases}
\displaystyle  \frac{\| f\|_{L^1(\R)}}{|x|} &\text{if } x\in [-\pi L, \pi L ] \setminus  [-\pi , \pi] \\
c_\alpha \| f\|_{\Co^\alpha} +  c\| f\|_\infty  &\text{if } x\in  [-\pi , \pi]
\end{cases} . 
\end{equation}

In order to obtain the first estimate,  observe that  if $x\in [-\pi L, \pi L] \setminus [-\pi, \pi]$, 
$ \displaystyle
\U_Lf(x) =  \int_{-\frac\pi2}^{\frac\pi2} \frac{f(t)}{2\pi L \tan(\frac{x-t}{2L})} dt
$
and then
\vspace*{-.3cm}
\[
\left| \U_Lf(x) \right| \le  \int_{-\frac\pi2}^{\frac\pi2} \frac{|f(t)|}{\pi |x-t|} dt \le 
\frac{\|f\|_{L^1(\R)}}{|x|} . 
\]
The second estimate in \eqref{Uest} follows from the fact if  $x\in [-\pi, \pi]$, we can decompose
\[
\U_Lf(x) = - \int_{-\pi }^{\pi } \frac{f(x)-f(t)}{2\pi L \tan(\frac{x-t}{2L})} dt 
+ f(x) \int_{\pi}^{\pi L}  \bigg( \frac{1}{  \tan(\frac{t-x}{2L})} - \frac{1}{\tan(\frac{t+x}{2L})}  \bigg)  \frac{dt}{2\pi L} . 
\]
On the one hand,  an explicit computation\footnote{Recall that $\frac{d}{dt} \log|\sin t| = \frac{1}{\tan t}$ for $t\in\T$, $t\neq 0$ and $\operatorname{supp}(f) \subseteq [-\frac \pi2, \frac \pi2]$.} gives for $x\in [-\frac\pi2, \frac \pi2]$,
\[ \begin{aligned}
\int_{\pi}^{\pi L}  \bigg( \frac{1}{  \tan(\frac{t-x}{2L})} - \frac{1}{\tan(\frac{t+x}{2L})}  \bigg)  \frac{dt}{2\pi L}
%&=  \log\left|\frac{\sin(\frac\pi2- \frac{x}{2L})}{\sin(\frac\pi2+ \frac{x}{2L})} \right|+ \log\left|\frac{\sin(\frac{x+\pi}{2L})}{\sin(\frac{x-\pi}{2L})} \right| \\
&=  \log\left|\frac{\sin(\frac{x+\pi}{2L})}{\sin(\frac{x-\pi}{2L})} \right|  \\
&= \log\left|\frac{x+\pi}{x-\pi} \right| + \O(L^{-2})
\end{aligned}\]
with a uniform error.  
On the other hand,
\[
\left|  \int_{-\pi }^{\pi } \frac{f(x)-f(t)}{2\pi L \tan(\frac{x-t}{2L})} dt  \right| \le
c_\alpha \| f\|_{\Co^\alpha}    
\qquad\text{where}\qquad
c_\alpha = \sup_{|x| \le \pi} \int_{-\pi }^{\pi } \frac{dt}{|x-t|^{1-\alpha}}  ,
\]
which shows that for any $x\in [-\pi, \pi]$, 
\[
\left| \U_Lf(x) \right| \le 
c_\alpha \|f\|_{\Co^\alpha} + c\|f\|_\infty . 
\]
Now, using formula \eqref{gL} and the estimate \eqref{Uest}, we obtain
\[\begin{aligned}
\| g_L^{(k)} \|_{L^1(\T)} 
& =   L^k \int_{-\pi}^{\pi} \left| \U_L(w^{(k)}) (xL) \right| dx 
= L^{k-1}  \int_{-\pi L}^{\pi L} \left| \U_L(w^{(k)}) (x) \right| dx 
\end{aligned}\]
which gives the bound  \eqref{gLest} by splitting the last integral in two parts.
\end{proof}

In order to identify the asymptotic variance in Theorem~\ref{clt:meso}, we also need the following easy consequence of Proposition~\ref{prop:U}.

\begin{corollary} \label{cor:variance}
According to the notations \eqref{var} and \eqref{sigma}, we have as $L\to+\infty$,
\[
\sigma^2(w_L)  \to \| w\|_{H^{1/2}(\R)}^2 . 
\]
\end{corollary}

\begin{proof}
By \eqref{UL}, it is immediate to verify that for any $x\in\R$,  $\U_L f (x) \to \Hi f (x)$ as $L\to+\infty$ where $\Hi$ denotes the Hilbert transform on $\R$. 
Now, by formula \eqref{sigmaU} and Proposition~\ref{prop:U},
\[
\sigma^2(w_L)  = \int_{-\pi}^{\pi} g_L(x)  w_L'(x)\frac{dx}{2\pi}
=  \int_{-\pi L}^{\pi L} \U_Lw(x)  w'(x)\frac{dx}{2\pi} . 
\]
 Since  $\operatorname{supp}(w) \subseteq [-\frac \pi2, \frac \pi2]$, by \eqref{Uest} the function $\U_Lw$ is uniformly bounded  and we conclude by the dominated convergence theorem that as $L\to+\infty$, 
\begin{equation} \label{norm}
\sigma^2(w_L)  = \int_{-\frac\pi2}^{\frac\pi2} \U_Lw(x)  w'(x)\frac{dx}{2\pi} 
\to \frac{1}{2\pi} \int_\R \Hi w(x) w'(x) dx . 
\end{equation}
It is well known that if $w\in \Co^1_c(\R)$, then the RHS of \eqref{norm} equals to $\|w\|_{H^{1/2}(\R)}^2$ which is also given by \eqref{var}. 
\end{proof}

Like in section~\ref{sect:globalclt},  our proof of Theorem~\ref{clt:meso}  relies on  Lemma~\ref{lem:Jo},  Proposition~\ref{prop:W} and the following proposition which provides a precise estimate for the term $R_0$ given by \eqref{W3}.

\begin{proposition} \label{prop:R}
Let $w$ be any function such that $g = \U w \in \Co^3(\T)$ and let $R_0(w)$ be given by \eqref{R0}. We have for any $0<\epsilon\le1$, 
\begin{equation} \label{R6}
R_0(w) \le \frac{\pi^3}{4} \left(  \epsilon^{-2} \|g\|_{L^1(\T)} + \epsilon^{-1}  \|g'\|_{L^1(\T)} + \frac{\epsilon}{3} \left(  \left\|  \mathrm{M}_\epsilon g''' \right\|_{L^1(\T)} +  \| g'\|_\infty  \right)   \right) ,
\end{equation}
where we denote  $ \displaystyle \mathrm{M}_\epsilon f(x)= \hspace{-.3cm} \sup_{|\zeta -x| \le \epsilon} \hspace{-.1cm}|f(\zeta)| $ for any $f\in \Co(\T)$. 
\end{proposition}

\begin{proof}
This is just a computation. Let us define 
\begin{align*} 
R_3(g)  =   \iint_{\T^2} \1_{|x_1-x_2| \ge \epsilon} \bigg| \frac{ g(x_1)-g(x_2)}{4 \sin^2(\frac{x_1-x_2}{2}) \tan(\frac{x_1-x_2}{2})}\bigg|\d x_1 \d x_2 ,
\qquad
R_4(g)   =  \iint_{\T^2} \1_{|x_1-x_2| \ge \epsilon} \bigg| \frac{ g'(x_1)+g'(x_2)}{4 \sin^2(\frac{x_1-x_2}{2})}\bigg|\d x_1 \d x_2 , 
\end{align*}
and 
\[
 R_5(g)   =  \iint_{\T^2} 
\1_{|x_1-x_2| \le \epsilon}
 \bigg|\frac{ g(x_1)-g(x_2) -  (g'(x_1) + g'(x_2)  \tan(\frac{x_1-x_2}{2})}{4 \sin^2(\frac{x_1-x_2}{2}) \tan(\frac{x_1-x_2}{2})}\bigg|\d x_1 \d x_2 .
\]
By \eqref{R0} and the triangle inequality, $R_0 \le R_3 + R_4+ R_5$, so it suffices to estimates each integral above individually.
Since $|\sin \vartheta | \ge \tfrac{2}{\pi} |\vartheta|$ for all $|\vartheta| \le \frac{\pi}{2}$, we have
\begin{equation} \label{R3}
 \begin{aligned}  
R_3(g)  & \le \frac{\pi^3}{4} \iint_{\T^2}  \1_{|x_1-x_2| \ge \epsilon} \bigg| \frac{ g(x_1)-g(x_2)}{(x_1-x_2)^3}\bigg| \d x_1\d x_2 \\
&\le \frac{\pi^3}{2} \int_\T |g(x_1)| \left( \int_{\T}    \1_{|x_1-x_2| \ge \epsilon}  \frac{\d x_2}{|x_2-x_1|^3} \right) \d x_1 \\
&\le\frac{\pi^3}{4 \epsilon^2} \| g\|_{L^1(\T)} . 
\end{aligned} 
\end{equation}
Similarly, we have
\begin{align} %\notag
R_4(g) %&\le \frac{\pi^2}{2} \int_\T |g'(x_1)| \left( \int_{\T}    \1_{|x_1-x_2| \ge \epsilon}  \frac{\d x_2}{|x_2-x_1|^2} \right) \d x_1 \\
\label{R4}
&\le\frac{\pi^2}{2 \epsilon } \| g'\|_{L^1(\T)} . 
\end{align}
In order to estimate $R_5$, since we assume that $g \in \Co^{3}(\T)$, by Taylor theorem, this implies that for any $x_1,x_2 \in\T$ with $|x_1-x_2| \le \epsilon$, we have
\begin{equation*}
\left| g(x_1)-g(x_2) -  (g'(x_1) + g'(x_2)  \tan(\tfrac{x_1-x_2}{2})\right|
 \le \frac{|x_1-x_2|^3}{6} \bigg( \sup_{|\zeta-x_1| \le \epsilon} \hspace{-.1cm} |g'''(\zeta)|  + \| g'\|_{\infty} \bigg) . 
\end{equation*}
Since  $g \in \Co^{3}(\T)$, the function  $\mathrm{M}_\epsilon g'''$ is  also continuous function on $\T$ and the previous estimate shows that 
\begin{align} \notag
 R_5(g)  & \le   \frac{\pi^3}{24}  \iint_{\T^2} 
\1_{|x_1-x_2| \le \epsilon} \left(  \mathrm{M}_\epsilon g'''(x_1) + \| g'\|_\infty \right) \d x_1 \d x_2 \\
& \label{R5}
\le  \frac{\pi^3 \epsilon}{12} \left(  \left\|  \mathrm{M}_\epsilon g''' \right\|_{L^1(
\T)} +  \| g'\|_\infty  \right) . 
\end{align}
Collecting the estimates \eqref{R3}--\eqref{R5}, since $R_0 \le R_3 + R_4+ R_5$, we obtain \eqref{R6}.
 \end{proof}

We are now ready to give the proof of our main result.

\begin{proof}[Proof of Theorem~\ref{clt:meso}]
Let the sequence $L=L(N)$ be as in the statement of the Theorem.
Recall that we assume that $ w\in \Co_c^{3+\alpha}(\R)$ for some $\alpha>0$ with $\operatorname{supp}(w) \subseteq [-\frac \pi2, \frac \pi2]$. Since $\| w_L ' \|_{L^1(\T)}$ is fixed for any $N\in\N$, by  Lemma~\ref{lem:Jo} and Proposition~\ref{prop:W}, we have  for all $N\ge N_\beta$, 
\begin{equation} \label{le1}
\bigg| \log \E^\beta_N\left[ e^{ \int w_L d\widetilde{\mu}_N} \right] - \frac{\sigma^2(w_L)}{\beta}  \bigg| 
\le C_\beta \left(   R_0 \log N+ R_1 +  R_2  N^{-5} \right)\frac{\log N}{N} ,
\end{equation}
where $R_0, R_1$ and $R_2$ are as in \eqref{R0} and \eqref{R12}. 
Then by Corollary~\ref{cor:variance}, in order to obtain Theorem~\ref{clt:meso}, it suffices to show that the RHS of \eqref{le1} converges to 0 as $N\to+\infty$. 

\smallskip
\underline{Estimate for $R_0$.}
By Proposition~\ref{prop:R} with $\epsilon=1/L$, since $\| g'_L\|_\infty \le \| w''_L\|_\infty$, we have 
\begin{equation*}
R_0(w_L) \le \frac{\pi^3}{4} \left(  L^2 \|g_L\|_{L^1(\T)} + L \|g_L'\|_{L^1(\T)} + \frac{1}{3L} \left(  \left\|  \mathrm{M}_\epsilon g_L''' \right\|_{L^1(\T)} +  \| w''_L\|_\infty  \right)   \right) . 
\end{equation*}
Observe that by \eqref{gL} and a change of variable:
\[ \begin{aligned}
 \mathrm{M}_\epsilon(g_L''')(x) &= \sup_{|\zeta -x| \le \epsilon} \hspace{-.1cm}|g_L'''(\zeta)|  = L^3   \sup_{|\zeta -x| \le L^{-1}} \hspace{-.1cm}|  \U_L(w^{(3)})(L\zeta)|  \\
 & = L^3  \mathrm{M}_1( \U_L(w^{(3)})) (xL) , 
\end{aligned} \]
%%% Abuse notation
so that 
\[
 \left\|  \mathrm{M}_\epsilon g_L''' \right\|_{L^1(\T)}= L^2 \int _{-\pi L}^{\pi L}
 \mathrm{M}_1( \U_L(w^{(3)})) (x) dx .
\]
Thus, since $\|w''_L\|_\infty = L^2 \| w''\|_{\infty}$,  using the estimate~\eqref{gLest}, this shows that 
\begin{equation*}
R_0(w_L) \le \frac{\pi^3}{4} \log(\pi L)   \left(r_{0,w}  +  r_{1,w} + \|w''\|_{\infty}  \right)L
+ \frac{L}{3} \int _{-\pi L}^{\pi L}
 \sup_{|\zeta -x| \le 1} \hspace{-.1cm}| \U_L(w^{(3)})(\zeta)| dx .
\end{equation*}
Using the estimate \eqref{Uest}, we see that the integral in the previous formula is bounded by  $2r_{3 ,w} \log(\pi L)$ where $r_{3 ,w}$ is as in Proposition~\ref{prop:gL}. 
Therefore, we obtain
\begin{equation} \label{R8}
R_0(w_L) \le 8 \log(\pi L)  \left(r_{0,w}  +  r_{1,w} + r_{3,w} + \|w''\|_{\infty}  \right) L .
\end{equation}

\smallskip
\underline{Estimate for $R_1$.}
By Proposition~\ref{prop:U}, it is easy to check that if $w\in \Co^2(\T)$ and $g = \U w$,  by the Cauchy--Schwartz inequality, 
\[ \begin{aligned}
 \| (gw')'\|_{L^1(\T)}  &\le  \|w'g'\|_{L^1(\T)} + \| gw'' \|_{L^1(\T)}  \\
 & \le \|w'\|_{L^2(\T)}^2 + \| w\|_{L^2(\T)} \|w'' \|_{L^2(\T)}  . 
 \end{aligned}\]
Thus, we obtain
\begin{equation} \label{R1}
R_1(w_L) \le  \sqrt{2\pi}   \| g_L''\|_{L^1(\T)} +  \|w_L'\|_{L^2(\T)}^2 + \| w_L\|_{L^2(\T)} \|w_L'' \|_{L^2(\T)}  . 
\end{equation}
Since $\| w_L^{(k)}\|_{L^2(\T)} = L^{k-1/2}\|w^{(k)}\|_{L^2(\R)}$ for $k=0,1,2$, by \eqref{gLest}, this shows that
\begin{equation*}
R_1(w_L) \le  \left(  \sqrt{2\pi} r_{2 ,w}   \log(\pi L)+    \|w'\|_{L^2(\R)}^2 + \| w\|_{L^2(\T)} \|w'' \|_{L^2(\T)}   \right)L. 
\end{equation*}

\smallskip
\underline{Estimate for $R_2$.}
Similarly, by Proposition~\ref{prop:U}, we check that if $w\in \Co^2(\T)$ and $g = \U w$, 
\[
R_2(w) =  3\| g'\|_{\infty} + \| g\|_\infty \| w'\|_\infty \le 3\| w''\|_{\infty} +   \| w'\|_{\infty}^2 .
\]
Since we assume that $L\le N$, this shows that for some universal constant $C>0$,
\begin{equation} \label{R2}
R_2(w_L) \le  C \| w''\|_{\infty} N^2 . 
\end{equation}

\smallskip
\underline{Conclusion.}
Collecting the estimates \eqref{R8}--\eqref{R2}, using the inequality \eqref{le1}, we have shown that there exists a constant $C_w>0$ which only depends on the test function $w$ such that
\begin{equation} \label{le2}
\bigg| \log \E^\beta_N\left[ e^{ \int w_L d\widetilde{\mu}_N} \right] - \frac{\sigma^2(w_L)}{\beta}  \bigg| \le  C_\beta \left( C_w   (\log N)^2L+    C \|w''\|_\infty N^{-3}  \right) \frac{\log N}{N} . 
\end{equation}
Hence, in the regime where $N^{-1} L(N) (\log N)^{3} \to 0$ as $N\to+\infty$,  the RHS of \eqref{le2} converges to 0. Moreover, since  $\sigma^2(w_L)  \to \| w\|_{H^{1/2}(\R)}^2$  by Corollary~\ref{cor:variance}, this completes the proof.
\end{proof}

\section{GMC applications} \label{sect:chaos}

\subsection{Proof of Theorem~\ref{thm:gmc1}} \label{sect:gmc1}

Recall that we let $\phi_r(\vartheta) = \log|1- r e^{\i\vartheta}|^{-1}$ for any $0\le r<1$ and that  for any $\vartheta\in\T$,
\begin{equation} \label{logcharpoly}
\log|P_N(r_N e^{\i\vartheta})| 
 = -  \textstyle{ \sum_{j=1}^N } \phi_{r_N}(\theta_j-\vartheta)
\end{equation}
 is a smooth linear statistic for a  test function which depends on $N\in\N$. 
The proof of  Theorem~\ref{thm:gmc1} relies directly on \cite[Theorem 1.7]{LOS18}. 
Recall that $\G$ denotes the GFF on $\T$ and let $\mathrm{P}_r(\theta) = 1+  2\sum_{k=1}^{+\infty} r^k \cos(k\theta)$  be the Poisson kernel for $\T$.  $\G$  is a Gaussian log--correlated field on $\T$ whose covariance is given by \eqref{Gcov} and we have for any $0\le r <1$ and all $\vartheta\in\T$,
\[
\G(r e^{\i\vartheta}) =  \int_\T \mathrm{P}_r(\theta-\vartheta) \G(e^{\i\theta}) \frac{d\theta}{2\pi} . 
\] 
Let $\mu_N^\gamma$ be as in \eqref{gmc1}. 
In order to apply \cite[Theorem 1.7]{LOS18}, we need to establish the following asymptotics:  for any $\beta>0$ and any  $n\in\N$,
\begin{equation} \label{modG1}
\log \E_{N}^\beta\left[ \exp\left(  {\textstyle \sum_{\ell=1}^n} \gamma_\ell \int  \phi_{r_\ell}(\theta - \vartheta_\ell) \mu_N(d\theta) \right)  \right]
= \frac 1\beta \sum_{\ell, k = 1}^n \gamma_\ell\gamma_k
\E\left[ \G(r_ke^{\i\vartheta_k})\G(r_\ell e^{\i\vartheta_\ell}) \right]
+\underset{N\to+\infty}{o(1)} , 
\end{equation}
uniformly for all $\boldsymbol{\vartheta} \in \T^n$, $0<r_1, \dots , r_n \le r_N$  and  $\boldsymbol{\gamma}$ in compact subsets of $\R^n$. Then,  this implies that for any $|\gamma| \le \sqrt{2\beta}$ and  any function $f\in L^1(\T)$: 
\[
\int f(\theta) \mu_N^\gamma(d \theta) \to \int f(\theta) \mu_\G^{\breve\gamma}(d \theta)
\]
in distribution as $N\to+\infty$. 
From this result, one can infer that for any $|\gamma| \le \sqrt{2\beta}$, the random measure $ \mu_N^\gamma$ converges in law with respect to the topology of weak convergence to the GMC measure $ \mu_\G^{\breve\gamma}$, see e.g.~\cite[Sect.~6]{Berestycki17}.

\medskip

In order to obtain  the mod--Gaussian asymptotics \eqref{modG1} and to prove Theorem~\ref{thm:gmc1}, 
let us observe that the test functions $\phi_{r_\ell}(\cdot - \vartheta_\ell)$ behave for $0<r_\ell<r_N$ like smooth mesoscopic linear statistics and we can therefore adapt our proof of Theorem~\ref{clt:global}.
Indeed, let us observe that 
\begin{equation} \label{phi}
\phi_r  =  \sum_{k\ge 1} \frac{r^k}{k} \cos(k\cdot) , \qquad 
w=  {\textstyle \sum_{\ell=1}^n} \gamma_\ell  \phi_{r_\ell}(\cdot - \vartheta_\ell) , 
\end{equation}
 then according to formula \eqref{sigma}, we have
\begin{align}
\sigma^2\left(w\right) 
&\notag
= 2\sum_{\ell , \ell' =1}^n  \gamma_\ell \gamma_{\ell'} \Re\left\{  \sum_{k=1}^{+\infty}  k (\widehat{ \phi_{r_\ell}})_k  \overline{ (\widehat{ \phi_{r_{\ell'}}})_k } e^{\i k (\vartheta_{\ell'} - \vartheta_\ell)} \right\} \\
&\notag
=  \frac 12\sum_{\ell , \ell' =1}^n  \gamma_\ell \gamma_{\ell'}\Re\left\{  \sum_{k=1}^{\infty} \frac{r_\ell^k r_{\ell'}^k}{k} e^{\i k (\vartheta_{\ell'} - \vartheta_\ell)}\right\} \\
&\notag
 =   \frac 12 \sum_{\ell , \ell' =1}^n \log| 1- r_\ell r_{\ell'} e^{\i  (\vartheta_{\ell'} - \vartheta_\ell)}  |^{-1} \\
&\label{sigmaG}
 =  \sum_{\ell , \ell' =1}^n \E\left[ \G(r_ke^{\i\vartheta_k})\G(r_\ell e^{\i\vartheta_\ell}) \right] , 
\end{align}
where we used  \eqref{Gcov} for the last step.
In the remainder of this section,  we will use the method of \emph{loop equation} -- in particular Lemma~\ref{lem:Jo} and Proposition~\ref{prop:W} -- to obtain the asymptotics \eqref{modG1}. 
First, according to \eqref{UT}, we have that the Hilbert transforms of the functions  \eqref{phi} are given by
\begin{equation} \label{psi}
\psi_r = \U \phi_r  =   \sum_{k\ge 1} \frac{r^k}{k} \sin(k\cdot) , \qquad
g= \U w = {\textstyle \sum_{\ell=1}^n} \gamma_\ell  \psi_{r_\ell}(\cdot - \vartheta_\ell) . 
\end{equation} 
Then, in order to control the error terms in Proposition~\ref{prop:W}, we need the following Lemmas. 
The proofs of Lemma~\ref{lem:phi1} and Lemma~\ref{lem:phi2} follow from routine computations. For completeness, the details are provided in the appendix~\ref{app:lem}. 

\begin{lemma} \label{lem:phi1}
There exists a universal constant $C>1$ such that the following estimates hold for any $0\le r <1$, 
\begin{align}
%\|\phi_r\|_{L^1(\T)} , \| \psi_r\|_{L^1(\T)}  & \le C \\
\label{est1}  \|\psi_r\|_\infty , \| \psi_r'\|_{L^1(\T)}  & \le C ,  \\
\label{est2} \| \phi_r \|_\infty , \| \phi_r'\|_{L^1(\T)} &\le -2\log(1-r) + C ,  \\
 \label{est3} \|\psi_r'\|_\infty , \|\phi'_r\|_\infty &\le \frac{1}{1-r}  ,
%\label{est4} \| \phi_r''\|_{\infty}, \| \psi_r''\|_{\infty}  &\le  \frac{1}{(1-r)^2} .
%\qquad\text{and}\qquad 
\end{align}
and also 
\begin{equation} \label{est5}
\| \phi_r''\|_{L^1(\T)} ,  \| \psi_r''\|_{L^1(\T)}  \le   \frac{C}{1-r} . 
\end{equation}
\end{lemma}

\begin{lemma} \label{lem:phi2}
Let $R_0$ be given by \eqref{R0}. There exists a universal constant $C>0$  such that for all $0 \le r<1$, 
\begin{equation} \label{R9}
R_0(\psi_r) , R_0(\phi_r) \le C  \frac{(\log(1-r))^2}{(1-r)\log \log  (1-r)^{-1}} . 
\end{equation}
\end{lemma}
 
 \medskip
 
We are now ready to give our Proof of Theorem~\ref{thm:gmc1}. 
We fix $n\in\N$, for any  $\boldsymbol{\gamma}\in \R^n$, $\boldsymbol{\vartheta} \in \T^n$, and $0<r_1, \dots , r_n \le r_N$, the function \eqref{phi} satisfies $w\in \Co^\infty(\T)$ and  by \eqref{est2}, we have
$\|w'\|_{L^1(\T)} \le  \eta = \| \boldsymbol{\gamma}\| (3 \log N +C)$ with $ \| \boldsymbol{\gamma}\| = \sum_{\ell=1}^n |\gamma_\ell|$. Hence, using the estimate \eqref{W6}, we obtain that there exists a constant $ C_{\beta,\boldsymbol{\gamma}}$ which depends only on $\|\boldsymbol{\gamma}\|$ and $\beta>0$ such that for all $N\ge N_\beta$ and all $t\in[0,1]$,
\begin{equation} \label{W7}
\E^\beta_{N, t w} \left[ | \widetilde{\W}_N | \right] 
 \le C_{\beta,\boldsymbol{\gamma}}  (\log N)^2 \left( R_2(w) N^{3-\eta^2/2\beta}  +R_0(w)  (\log N)^2 + R_1(w)  \right) , 
\end{equation}
where $R_0, R_1, R_2$ are given by \eqref{R12} and \eqref{R0} with $g = \U w=  {\textstyle \sum_{\ell=1}^n} \gamma_\ell  \psi_{r_\ell}(\cdot - \vartheta_\ell) $. 
In particular, we have $\| g\|_\infty \le  C \|\boldsymbol{\gamma}\| $ by \eqref{est1},  and using the estimate \eqref{est3}, we see that there exists a constant $C _{\boldsymbol{\gamma}}$ which only depends on $\|\boldsymbol{\gamma}\|$ and $n\in\N$ such that
%%% Abuse notation
\[\begin{aligned}
R_2(w)  & =  \| g'\|_{\infty} + \| g\|_\infty \| w'\|_\infty \\
&\le \|\boldsymbol{\gamma}\|  \max_{\ell=1,\dots, n}  \|  \psi_{r_\ell}' \|_\infty
+ C\|\boldsymbol{\gamma}\|^2  \max_{\ell=1,\dots, n}  \|  \phi_{r_\ell}' \|_\infty \\
&\le  \frac{C_{\boldsymbol{\gamma}}}{1-r_N} .
\end{aligned}\]
Similarly, 
$\| w'\|_\infty \le   \|\boldsymbol{\gamma}\| (1-r_N)^{-1} $ by \eqref{est3} and, with a possibly different constant  $C _{\boldsymbol{\gamma}}$, we deduce from \eqref{est1} and \eqref{est5} that
\[\begin{aligned}
R_1(w)  & =   \| g''\|_{L^1(\T)} + \| (gw')'\|_{L^1(\T)}  \\
&\le \|\boldsymbol{\gamma}\|   {\textstyle \sum_{\ell=1}^n} \left(   \| \psi_{r_\ell}''\|_{L^1(\T)} + \|g\|_{\infty} \| \phi_{r_\ell}''\|_{L^1(\T)} +  \|w'\|_{\infty} \| \psi_{r_\ell}'\|_{L^1(\T)}   \right) \\
% &\le   n\|\boldsymbol{\gamma}\|  \frac{C + C^2\|\boldsymbol{\gamma}\| +  C\|\boldsymbol{\gamma}\|} {1-r_N}
& \le \frac{C_{\boldsymbol{\gamma}}}{1-r_N} . 
\end{aligned}\]
Since $\eta = \| \boldsymbol{\gamma}\| (3 \log N +C)$, this shows that the first term on the RHS of \eqref{W7} is negligible compared to the third term and we obtain for all $N\ge N_\beta$, 
\begin{equation} \label{W8}
\delta_N(w)  \le \frac{2C_{\beta,\boldsymbol{\gamma}} }{\beta N}  (\log N)^2 \left( R_0(w)  (\log N)^2 + \frac{2C_{\boldsymbol{\gamma}}}{1-r_N} \right) .
\end{equation}

Then, by \eqref{R0}, since 
$\displaystyle R_0(w) \le \| \boldsymbol{\gamma}\| \max_{\ell=1,\dots, n} R_0(\psi_{r_\ell})$, we deduce from Lemma~\ref{lem:phi2} that 
\[
R_0(w)  \le C \| \boldsymbol{\gamma}\| \frac{(\log(1-r_N))^2}{(1-r_N)\log \log  (1-r_N)^{-1}} .
\]
By \eqref{W8}, since $r_N = 1- \frac{(\log N)^6}{N}$, this implies that 
\begin{equation*}
\delta_N(w)  \le \frac{2C_{\beta,\boldsymbol{\gamma}} }{\beta}   \left(  \frac{C  \| \boldsymbol{\gamma}\|}{\log \log\left( N / (\log N)^6 \right)} + \frac{2C_{\boldsymbol{\gamma}}}{(\log N)^4}  \right) . 
\end{equation*}
According to Lemma~\ref{lem:Jo} since $(\widehat{\phi_r})_0 =0$ for any $r\in [0,1]$,  this proves that 
uniformly for all $\boldsymbol{\vartheta} \in \T^n$, $0<r_1, \dots , r_n \le r_N$ and $\boldsymbol{\gamma}\in \R^n$, we have  for all $N\ge N_\beta$, 
\begin{equation} \label{W9}
\left| \log \E^\beta_N\left[ e^{ \int w d\mu_N} \right] - \beta^{-1}\sigma^2(w) \right| \le \frac{2C_{\beta,\boldsymbol{\gamma}} }{\beta}   \left(  \frac{C  \| \boldsymbol{\gamma}\|}{\log \log\left( N / (\log N)^6 \right)} + \frac{2C_{\boldsymbol{\gamma}}}{(\log N)^4}  \right) .
\end{equation}
Since  $w=  {\textstyle \sum_{\ell=1}^n} \gamma_\ell  \phi_{r_\ell}(\cdot - \vartheta_\ell)$ and the RHS of \eqref{W9} converges to 0 as $N\to+\infty$,  by formula \eqref{sigmaG}, we obtain the asymptotics \eqref{modG1}. Whence, we deduce from \cite[Theorem 1.7]{LOS18} that  for any $|\gamma|< \sqrt{2\beta}$, the random measure $\mu_N^\gamma$ converges in law with respect to the topology of weak convergence to the GMC measure $ \mu_\G^{\breve\gamma}$,

\subsection{Proof of Theorem~\ref{thm:gmc2}} \label{sect:gmc2}

The proof of Theorem~\ref{thm:gmc2} is almost identical to that of Theorem~\ref{thm:gmc1} in the previous section, so we just go through the argument quickly. 
According to  \eqref{imcharpoly} and \eqref{psi}, we have for any $0\le r<1$ and all $\vartheta\in\T$,
\[
\Psi_{N,r}(\vartheta)=  \textstyle{ \sum_{j=1}^N }  \psi_{r}(\vartheta - \theta_j) .
\] 
 We claim that for any $\beta>0$ and any  $n\in\N$,
\begin{equation} \label{modG2}
\log \E_{N}^\beta\left[ \exp\left(  {\textstyle \sum_{\ell=1}^n} \gamma_\ell \Psi_{N, r_\ell}( \vartheta_\ell)  \right)  \right]
= \frac 1\beta \sum_{\ell, k = 1}^n \gamma_\ell\gamma_k
\E\left[ \G(r_ke^{\i\vartheta_k})\G(r_\ell e^{\i\vartheta_\ell}) \right]
+\underset{N\to+\infty}{o(1)} , 
\end{equation}
uniformly for all $\boldsymbol{\vartheta} \in \T^n$, $0<r_1, \dots , r_n \le r_N$  and  $\boldsymbol{\gamma}$ in compact subsets of $\R^n$.
Hence,  by applying \cite[Theorem~1.7]{LOS18}, we obtain for any $|\gamma| \le \sqrt{2\beta}$ the random measure $\widetilde{\mu}_N^\gamma$  given by \eqref{gmc2} converges in law with respect to the topology of weak convergence to the GMC measure $ \mu_\G^{\breve\gamma}$ associated with the GFF on $\T$. 
The proof of the asymptotics \eqref{modG2} is analogous to that of \eqref{modG1}.
Namely,  we have ${\textstyle \sum_{\ell=1}^n} \gamma_\ell \Psi_{N, r_\ell}( \vartheta_\ell) = \int g d\mu_N  $ where the function $g\in \Co^\infty(\T)$ is given by \eqref{psi}.
By \eqref{est1}, we have $\|g'\|_{L^1(\T)} \le C \| \boldsymbol{\gamma}\|$, so that by directly applying Lemma~\ref{lem:Jo}  and Proposition~\ref{prop:W},  we obtain for all $N\ge N_\beta$, 
\begin{equation} \label{est11}
\left| \log \E^\beta_N\left[ e^{ \int g d\mu_N} \right] - \beta^{-1}\sigma^2(g) \right| 
\le \delta_N(g) 
\le C_\beta \left(   R_0(g) \log N+ R_1(g) +  R_2(g)  N^{-5} \right)\frac{\log N}{N} . 
\end{equation}
Going through the estimates of section~\ref{sect:gmc1}, since the Hilbert transform of  $g$ is given by $\U g  = -w$, we have
\[
R_1(g), R_2(g) \le C_{\boldsymbol{\gamma}}  \frac{\log(1-r_N)^{-1}}{1-r_N} 
\qquad\text{and}\qquad
R_0(g)  \le C \| \boldsymbol{\gamma}\| \frac{(\log(1-r_N))^2}{(1-r_N)\log \log  (1-r_N)^{-1}} . 
\]
These estimates show that with $r_N = 1- \frac{(\log N)^4}{N}$, 
\[
\delta_N(g)  \le C_\beta \left(  \frac{ C \| \boldsymbol{\gamma}\|}{\log \log  (1-r_N)^{-1}} + \frac{2C_{\boldsymbol{\gamma}}}{(\log N)^2} \right) , 
\]
so that the LHS of \eqref{est11} converges to 0 as $N\to+\infty$. 
By definition of the Hilbert transform, $\sigma^2(g) = \sigma^2(w)$  is given by \eqref{sigmaG}. Hence, 
since ${\textstyle \sum_{\ell=1}^n} \gamma_\ell \Psi_{N, r_\ell}( \vartheta_\ell) = \int g d\mu_N$, we obtain the asymptotics \eqref{modG2} and this completes the proof of Theorem~\ref{thm:gmc2}.

\subsection{Thick points: Proofs of Proposition~\ref{prop:TP1} and Proposition~\ref{prop:TP2}} \label{sect:TP}

The goal of this section is to deduce from Theorem~\ref{thm:gmc1} some important properties of the \emph{thick points} of the characteristic polynomial of the \Cbeta. 
Recall that for any $\gamma>0$, the set of $\gamma$--thick points of the  characteristic polynomial  is 
\[
\mathscr{T}^\gamma_N = \left\{ \theta \in \T :  | P_N(e^{\i\theta})|   \ge N^{\gamma/\beta} \right\} .
\]
The connection between Theorem~\ref{thm:gmc1} and Proposition~\ref{prop:TP1} comes from the fact that the random measure $\mu_N^\gamma$ is essentially supported on $\mathscr{T}^\gamma_N $ for large $N\in\N$, see e.g. \cite[section~3]{CFLW}. 
In the following, we relie on this heuristic to obtain a lower for the Lebesgue measure $|\mathscr{T}^\gamma_N|$ when  $\gamma$ is less than the critical value $\gamma_* = \sqrt{2\beta}$. 
Then, we obtain the complementary upper--bound by using a result of Su \cite[Theorem 1.2]{Su09} -- see Lemma~\ref{lem:UB} below.
By combining these estimates, the proof of  Proposition~\ref{prop:TP1} will be given at the end of this section after we obtained the necessary lemmas. 
Since the proof of Proposition~\ref{prop:TP2} is almost identical to that of  Proposition~\ref{prop:TP1}, we skip it and only comment on the main differences in Remark~\ref{rk:imagpart} below.

\medskip

We let for any $N\in\N$,  $0<r<1$ and $\theta\in\T$, 
\begin{equation} \label{Upsilon}
\Upsilon_{N,r}(\theta) = \log|P_N(r e^{\i \theta})|
\end{equation}
and $\displaystyle \Upsilon_{N}(\theta) = \lim_{r\to1} \Upsilon_{N,r}(\theta) = \log|P_N(e^{\i \theta})|$.
Recall that $r_N = 1- \frac{(\log N)^6}{N}$. 
Observe that it follows immediately from the asymptotics \eqref{modG1} and formula  \eqref{Gcov} that there exists a constant $R_\beta >1 $ such that for all $|\gamma| \le 2 \gamma_*$ and $\theta \in \T$,
\begin{equation}\label{Rest}
 R_\beta^{-1}  (1-r_N)^{-\gamma^2/2\beta} \le \E^\beta_N \left[ e^{\gamma \Upsilon_{N,r_N}(\theta)} \right] \le R_\beta   (1-r_N)^{-\gamma^2/2\beta} .
\end{equation}

The following result follows essentially from \cite[Proposition 3.8]{CFLW}. Since our context is slightly different, we provide the main steps of the proof for completeness.

\begin{lemma} \label{lem:TP}
Let $\epsilon_N =\frac{(\log N)^6}{N}$ and $r_N= 1-\epsilon_N$.  Fix $\gamma >0$ and define
\[
\widetilde{\mathscr{T}}_N^{\gamma}  = \left\{ \theta \in \T :  \Upsilon_{N,r_N}(\theta) \ge \frac{\gamma}{\beta} \log N  \right\} . 
\]
  For any $\delta>0$ such that $\gamma+\delta< \gamma_*$, we have for any $C>0$ as $N\to+\infty$,
\[
\P_{N}^\beta\left[  |\widetilde{\mathscr{T}}_N^{\gamma} |  \le C N^{-(\gamma+\delta)^2/2\beta} \right] \to0.
\]
\end{lemma}

\begin{proof}
Let us fix small $\epsilon, \delta>0$. 
Observe that by definition of the random measure $\mu_N^\gamma$, \eqref{gmc1}, by using the estimate \eqref{Rest}, we obtain
\[
\mu_N^{\gamma+\delta}\left(\widetilde{\mathscr{T}}_N^{\gamma}  \setminus\widetilde{\mathscr{T}}_N^{\gamma+2\delta}\right) \le R_\beta \epsilon_N^{(\gamma+\delta)^2/2\beta} N^{(\gamma+\delta)(\gamma +2\delta)/\beta}
|\widetilde{\mathscr{T}}_N^{\gamma} |
 = R_\beta (\log N)^{6\gamma} N^{(\gamma+2\delta)^2/2\beta  - \delta^2/2\beta} |\widetilde{\mathscr{T}}_N^{\gamma} | . 
\]  
This shows that if $N$ is sufficiently large,
\[\begin{aligned}
\P^\beta_N\left[  |\widetilde{\mathscr{T}}_N^{\gamma} |  \le C N^{-(\gamma+2\delta)^2/2\beta } \right] 
&\le  \P^\beta_N\left[ \mu_N^{\gamma+\delta}\left(\widetilde{\mathscr{T}}_N^{\gamma}  \setminus\widetilde{\mathscr{T}}_N^{\gamma+2\delta}\right) \le C  R_\beta (\log N)^{6\gamma} N^{ - \delta^2/2\beta} \right] \\
&\le  \P^\beta_N\left[  \mu_N^{\gamma+\delta}\left(\T\right) \le 3\epsilon\right]  + 
\P^\beta_N\left[ \mu_N^{\gamma+\delta}\left(\T\setminus \widetilde{\mathscr{T}}_N^{\gamma} \right)\ge \epsilon\right]
+\P^\beta_N \left[\mu_N^{\gamma+\delta}\left(\widetilde{\mathscr{T}}_N^{\gamma+2\delta}\right) \ge \epsilon \right] . 
\end{aligned}\] 
Moreover, by  \cite[Lemma 3.2]{CFLW}, we also have the estimates: 
\[
\E^\beta_N \left[\mu_N^{\gamma+\delta}\left(\widetilde{\mathscr{T}}_N^{\gamma+2\delta}\right)  \right] , 
\E^\beta_N \left[\mu_N^{\gamma+\delta}\left(\T\setminus \widetilde{\mathscr{T}}_N^{\gamma} \right) \right]
\le R_\beta \epsilon_N^{\delta^2/2\beta}  . 
\]
Since,  by Theorem~\ref{thm:gmc1}, the random variable $\mu_N^{\gamma+\delta}(\T)$ converges in law to $\mu_{\G}^{\widetilde{\gamma+\delta}}(\T)$, this implies that
\[
\limsup_{N\to+\infty}\P^\beta_N\left[  |\widetilde{\mathscr{T}}_N^{\gamma} |  \le N^{-(\gamma+2\delta)^2/2\beta} \right] 
\le  \P^\beta_N\left[  \mu_{\G}^{\widetilde{\gamma+\delta}} \left(\T\right) \le 3\epsilon\right] . 
\]
Since this estimate holds for arbitrary small $\epsilon>0$ and the random variable $\mu_{\G}^{\widetilde{\gamma+\delta}}(\T)>0$ almost surely\footnote{This fact follows e.g. from the construction of the random measure $\mu_{\G}^\gamma$ in \cite{Berestycki17} -- see also \cite[Proposition 2.1]{CFLW} for further details.} for any $\gamma < \gamma_* -\delta$, this completes the proof (we may also replace $2\delta$ by $\delta$ since $\delta>0$ is arbitrary small.)
\end{proof}

\begin{lemma}[Upper--bounds] \label{lem:UB}
For any $\gamma >0$ and any small $\delta>0$, we have
\begin{equation} \label{UB2}
\P_{N}^\beta\left[  |\mathscr{T}_N^{\gamma} |  \ge C N^{-\gamma^2/2\beta+ \delta} \right] \to0.
\end{equation}
Moreover, we have for any  small $\delta>0$, 
\begin{equation} \label{UB0}
\lim_{N\to+\infty}\P^\beta_N\left[ \max_{\theta\in\T} \log|P_N(e^{\i\theta})| \le (
1 +\delta)  \sqrt{\tfrac{2}{\beta}}\log N\right] = 1.
\end{equation}
\end{lemma}

\begin{proof}
These estimates follow by standard arguments using the so--called \emph{first moment method} and the explicit formula for the $\gamma$ moments of $|P_N|$. 
By  \cite[Theorem 1.2]{Su09} case (1) that for any $\gamma>-1$ and $\vartheta\in\T$, 
\[
 \E^\beta_N\left[ |P_N(e^{\i \vartheta})|^\gamma \right] = \prod_{k=0}^{N-1} \frac{\Gamma(1+ \frac{\beta k}{2})\Gamma(1+\gamma+ \frac{\beta k}{2})}{\Gamma(1+ \frac{\beta k+\gamma}{2})^2} . 
\]
By using e.g. the asymptotics of \cite[Theorem~5.1]{DHR18},  this formula implies that  there exists a constant $C_\beta>0$ such that for all $\gamma \in [0,\gamma_*] $,
\begin{equation} \label{UB3}
 \E^\beta_N\left[ |P_N(e^{\i \vartheta})|^\gamma \right]  \le C_\beta N^{\gamma^2/2\beta} .  
\end{equation}
Observe that by definition of the set $\mathscr{T}^\gamma_N$ and Markov's inequality, this estimate implies that for any $\gamma \in [0,\gamma_*]$, 
\[\begin{aligned}
\E_{N}^\beta\left[  |\mathscr{T}_N^{\gamma} |\right] 
&= \int_\T \P_{N}^\beta\left[    | P_N(e^{\i\theta})|   \ge N^{\gamma/\beta}  \right]\d\theta  \\
&\le N^{-\gamma^2/\beta} \int_\T\E^\beta_N\left[ |P_N(e^{\i \vartheta})|^\gamma \right] d\theta\\
&\le  2\pi C_\beta N^{-\gamma^2/2\beta} . 
\end{aligned}\]
By Markov's inequality, this immediately implies \eqref{UB2}. 
In order to prove the second claim, we use that by \cite[Lemma~4.3]{CMN18}, since $P_N$ is a polynomial of degree $N$, we have the deterministic bound:  $\max_\T |P_N| \le 14 \max_{k=1, \dots 2N} |P_N(e^{\i 2\pi k/2N})|$. 
This implies that for any $\delta>0$, we have if $N$ is sufficiently large,
\begin{equation} \label{UB1}
\P^\beta_N\left[ \max_{\theta\in\T} \log|P_N(e^{\i\theta})| \ge (1 +\delta)  \sqrt{\tfrac{2}{\beta}}\log N\right] \le \P^\beta_N\bigg[ \max_{k=1, \dots 2N} |P_N(e^{\i 2\pi k/2N})|  \ge N^{\sqrt{\frac{2}{\beta}}(1 + \delta/2)  }\bigg]  . 
\end{equation}
By a union bound, Markov's inequality and using the estimate \eqref{UB3} with $\gamma =\gamma_*= \sqrt{2\beta}$, we obtain  if $N$ is sufficiently large,
\[ \begin{aligned}
\P^\beta_N\left[ \max_{\theta\in\T} \log|P_N(e^{\i\theta})| \ge (1 +\delta)  \sqrt{\tfrac{2}{\beta}}\log N\right] &\le N^{-\gamma_*\sqrt{\frac{2}{\beta}}(1 + \delta/2)  }  \sum_{k=1, \dots 2N}   \E^\beta_N\left[ |P_N(e^{\i \vartheta})|^{\gamma_*} \right] \\
&\le N^{-2(1+\delta/2)}\cdot 2 C_\beta N^2
= 2 C_\beta N^{-\delta} . 
\end{aligned}\]
 This yields \eqref{UB0}. 
\end{proof}

Recall that for $0<r<1$, $\mathrm{P}_r(\cdot) = 1+  2\sum_{k=1}^{+\infty} r^k \cos(k\cdot) = \frac{1-r^2}{1+r^2 -2r \cos(\cdot)}$ denotes the Poisson kernel. 
Since the function $\Upsilon_N = \log|P_N|$ is harmonic in $\D$, according to \eqref{Upsilon}, we have for any $0<r<1$ and $x\in\T$, 
\begin{equation} \label{UB5}
\Upsilon_{N,r}(x)  =  \int_{\T} \Upsilon_N(\theta)  \mathrm{P}_r(\theta-x) \frac{d\theta}{2\pi} . 
\end{equation}
Using that $\Upsilon_{N,r}$ is  the convolution of $\Upsilon_N$ with a smooth probability density function, we can deduce from Lemma~\ref{lem:TP} and Lemma~\ref{lem:UB} a lower--bound for the Lebesgue measure of the set $\gamma$--thick points of $|P_N| = e^{\Upsilon_N}$.

\begin{proposition}[Lower--bounds] \label{prop:LB}
Fix $\gamma , \beta >0$ and let $\mathscr{T}_N^{\gamma} $ be as in \eqref{TP2}. 
For any small $0<\delta <\gamma$ such that $\gamma+\delta< \gamma_*$, we have as $N\to+\infty$, 
\begin{equation} \label{UB4}
\P^\beta_N\left[  |\mathscr{T}_N^{\gamma} |  \le N^{-(\gamma+\delta)^2/2\beta} \right] \to 0 .
\end{equation}
In particular, we have for any  small $\delta>0$, 
\begin{equation} \label{LB0}
\lim_{N\to+\infty}\P^\beta_N\left[ \max_{\theta\in\T} \log|P_N(e^{\i\theta})| \ge (
1 -\delta)  \sqrt{\tfrac{2}{\beta}}\log N\right] = 1.
\end{equation}

\end{proposition}

\begin{proof}
Let us fix $0<\delta <\gamma$ such that $\gamma+\delta< \gamma_*$ and define the event 
\[
\mathscr{A}_N = \{ \max_\T \Upsilon_N  \le  (\gamma_* + \tfrac \delta 2) \log N  \} . 
\]
By Lemma~\ref{lem:UB}, we have $\P^\beta_N[\mathscr{A}_N] \to 1$ as $N\to+\infty$.
Let us choose $L>0$ which only depends on the parameters $\delta, \beta>0$ such that
\begin{equation} \label{Poissonconc}
\int_{|\theta| \ge L (1-r)} \mathrm{P}_r(\theta) \frac{d\theta}{2\pi} \le \frac{\delta}{2\gamma_*+\delta} .
\end{equation}
By \eqref{UB5}, let us observe that since $\mathrm{P}_r$ is a smooth probability density function, conditionally on $\mathscr{A}_N$, we have  for any $0<r<1$ and $x\in\T$, 
\[ \begin{aligned}
\Upsilon_{N,r}(x) 
& =  \int_{\T \setminus \mathscr{T}_N^\gamma}\Upsilon_N(\theta)  \mathrm{P}_r(\theta-x) \frac{d\theta}{2\pi} +   \int_{\mathscr{T}_N^\gamma}\Upsilon_N(\theta)  \mathrm{P}_r(\theta-x) \frac{d\theta}{2\pi}  \\
&\le \frac{\gamma}{\beta} \log N 
+ \int_{|\theta -x| \ge L (1-r)} \hspace{-.5cm}\Upsilon_N(\theta)  \mathrm{P}_r(\theta-x) \frac{d\theta}{2\pi}  
+  \int_{\mathscr{T}_N^\gamma \cap |\theta -x| \le L (1-r) } \hspace{-.5cm}\Upsilon_N(\theta)  \mathrm{P}_r(\theta-x) \frac{d\theta}{2\pi}   \\
&  \le \left(\frac{\gamma}{\beta} + \frac \delta2  \right) \log N +  (\gamma_* + \tfrac \delta 2) 
\left( \int_{\mathscr{T}_N^\gamma \cap |\theta -x| \le L (1-r) } \hspace{-.5cm} \mathrm{P}_r(\theta-x) \frac{d\theta}{2\pi} \right) \log N  
\end{aligned} \]
where we used that $\max_\T \Upsilon_N \le (\gamma_* + \frac{\delta}{2}) \log N$ conditionally on $\mathscr{A}_N$ and  \eqref{Poissonconc} at the last step.
Since $\mathrm{P}_r(\theta) \le \frac{2}{1-r}$ for all $\theta \in\T$, this implies that 
\begin{equation} \label{TP3}
\Upsilon_{N,r_N}(x) 
\le \left(\frac{\gamma}{\beta} + \frac \delta2  \right) \log N +  \frac{3\gamma_*  \log N}{\epsilon_N} \big| \theta \in \mathscr{T}_N^\gamma : |\theta -x| \le L \epsilon_N \big|  .
\end{equation}

Choosing $L$ possibly larger, let us assume that $ M = \frac{\pi}{ L\epsilon} \in \N$ and for $k=1, \dots , M$, we choose $x_k \in [\frac{2\pi(k-1)}{M} , \frac{2\pi k}{M}  ]$ such that 
\[
\Upsilon_{N,r_N}(x_k) = \max_{\theta \in [\frac{2\pi(k-1)}{M} , \frac{2\pi k}{M}  ]} \Upsilon_{N,r_N}(\theta) . 
\]
Then we obviously have
\begin{equation} \label{TP4}
| \widetilde{\mathscr{T}}_N^{\gamma+\delta}  | \le \frac {2\pi}{M}  \sum_{k=1, \dots, M} \1_{\{ x_k \in \widetilde{\mathscr{T}}_N^{\gamma+\delta} \}} .
\end{equation}
Moreover,  if $x_k \in \widetilde{\mathscr{T}}_N^{\gamma+\delta}  $, then using the estimate \eqref{TP3}, 
\[
 \big| \theta \in \mathscr{T}_N^\gamma : |\theta -x_k| \le L \epsilon_N \big| 
  \ge \frac{\epsilon_N \delta}{6 \gamma_*}  
 = \frac{\delta}{12L\gamma_*} \frac{2\pi}{M}.  
\] 
By \eqref{TP4}, this implies that  conditionally on $\mathscr{A}_N$: 
\begin{equation}  \label{TP5}
| \widetilde{\mathscr{T}}_N^{\gamma+\delta}  | \le  \frac{\delta}{12L\gamma_*}  \sum_{k=1, \dots, M} \1_{\{ x_k \in \widetilde{\mathscr{T}}_N^{\gamma} \}} \big| \theta \in \mathscr{T}_N^\gamma : |\theta -x_k| \le L \epsilon \big|  . 
\end{equation}
Then observe that since $x_k \in [\frac{2\pi(k-1)}{M} , \frac{2\pi k}{M}  ]$ and the intervals $[\frac{2\pi(k-1)}{M} , \frac{2\pi k}{M}  ]$ are disjoints (expect for the endpoints), we have 
\begin{equation}  \label{TP7}
\sum_{k=1, \dots, M}  \big| \theta \in \mathscr{T}_N^\gamma : |\theta -x_k| \le L \epsilon_N \big| 
\le 2 |\mathscr{T}_N^\gamma| . 
\end{equation}
Using the bounds \eqref{TP5} and \eqref{TP7}, we obtain that conditionally on $\mathscr{A}_N$, 
\[
| \widetilde{\mathscr{T}}_N^{\gamma+\delta}  | \le \frac{\delta}{24L\gamma_*}  |\mathscr{T}_N^\gamma|
\]
This shows that 
\begin{equation}  \label{TP6}
\P^\beta_N\left[  |\mathscr{T}_N^{\gamma} |  \le N^{-(\gamma+2\delta)^2/2\beta} \right] \le 
\P^\beta_N\left[  | \widetilde{\mathscr{T}}_N^{\gamma+2\delta}  |  \le \tfrac{24L\gamma_*}{\delta} N^{-(\gamma+2\delta)^2/2\beta} \right]  + \P^\beta_N[\mathscr{A}_N^c] .
\end{equation}
By Lemma~\ref{lem:TP}, the first term on the RHS of \eqref{TP6} converges to 0 and, by Lemma~\ref{lem:UB}, we also have $\P^\beta_N[\mathscr{A}_N^c] \to 0$ as $N\to+\infty$. 
This completes the proof of \eqref{UB4} (since $\delta>0$ is arbitrary small,  we may replace $2\delta$ by $\delta$ in the end). 
In particular, this shows that the sets $\mathscr{T}_N^{\gamma}$ are non-empty for all $0\le \gamma < \gamma_* = \sqrt{2\beta}$ since they have positive Lebesgue measure. 
This implies the lower--bound \eqref{LB0}. 
\end{proof}

It is now straightforward to complete our proof of Proposition~\ref{prop:TP1}. 

\begin{proof}[Proof of Proposition~\ref{prop:TP1}]
Since the estimates \eqref{UB2} and  \eqref{UB4} hold for any small $\delta>0$, we obtain that any $\gamma \in [0, \gamma_*)$,  
$\displaystyle\frac{\log|\mathscr{T}^\gamma_N |}{ \log N} \to - \frac{\gamma^2}{2\beta}$
 in probability as $N\to+\infty$. 
 Moreover, by combining the estimates \eqref{UB0} and \eqref{UB0}, we also obtain the claim \eqref{max1} for the maximum of $\log|P_N|$. 
\end{proof}

\begin{remark} \label{rk:imagpart}
The proof of Proposition~\ref{prop:TP2} follows from similar arguments. 
In particular, by  Theorem~\ref{thm:gmc2}, we obtain the counterpart of Lemma~\ref{lem:TP} for the thick points of the field $\Psi_{N,r_N}$, \eqref{imcharpoly}. Since we  have for any $0\le r<1$ and $x\in\T$, 
\[
\Psi_{N,r}(x)  =  \int_{\T} \Psi_N(\theta)  \mathrm{P}_r(\theta-x) \frac{d\theta}{2\pi} ,
\]
 by going through the proof of Proposition~\ref{prop:LB}, we obtain that for any small $0<\delta <\gamma$ such that $\gamma+\delta< \gamma_*$, 
 \begin{equation*} 
\lim_{N\to+\infty}\P^\beta_N\left[  |\widetilde{\mathscr{T}}_N^{\gamma} |  \le N^{-(\gamma+\delta)^2/2\beta} \right] \to 0 .
\end{equation*}
 The complementary upper--bound for $  |\widetilde{\mathscr{T}}_N^{\gamma} |$ is obtained by the \emph{first moment method} as in Lemma~\ref{lem:UB} using the asymptotics from \cite[Theorem 1.2]{Su09} case (2): for any $|\gamma|<2$ and $\theta\in\T$,
\begin{equation} \label{ZG1}
\E^\beta_N \big[ e^{\gamma \Psi_N(\theta)} \big]  =  \prod_{k=0}^{N-1}  \frac{\Gamma(1+ \frac{k\beta}{2})^2}{\Gamma(1+ \frac{k\beta+\i \gamma}{2}) \Gamma(1+ \frac{k\beta-\i \gamma}{2})} . 
\end{equation}
To obtain these asymptotics, one must take $s= -\i \gamma/2$ in \cite[Theorem 1.2]{Su09} and observe that according to formula \eqref{Psi}, we have for all $\vartheta \in \T \setminus \{ \theta_k\}_{k=1}^N $, 
\begin{equation} \label{Psi2}
\Psi_N(\vartheta) =\Im \log\left( {\textstyle \prod_{j=1}^n }(1- e^{\i(\vartheta-\theta_j)}) \right) . 
\end{equation}
To obtain \eqref{Psi2}, it suffices to observe e.g. that $\Psi_N = \U \Upsilon_N$ where $\Upsilon_N = \Re \log P_N$ and $\U$ is the Hilbert transform.
Moreover, in order to obtain the upper--bound for $\max_\T \Psi_N$, one cannot use the estimate \eqref{UB1} as in Lemma~\ref{lem:UB}. Instead, since $\psi(\theta) = \frac{\theta-\pi}{2} $ for all $\theta \in (0, 2\pi)$, by formula \eqref{Psi}, we have the deterministic bound:
\begin{equation} \label{UB8}
 \Psi_N(\vartheta)  \le \Psi_N(\tfrac{2\pi(k-1)}{N}) + \pi  , \qquad \vartheta \in[\tfrac{2\pi (k-1)}{N}, \tfrac{2\pi k}{N}] , 
 \end{equation}
 for $k=1, \dots,  N$. So we can just replace the estimate \eqref{UB1} by 
 \[
 \P^\beta_N\left[ \max_{\theta\in\T} \Psi_N(\theta)\ge (1 +\delta)  \sqrt{\tfrac{2}{\beta}}\log N\right] \le \P^\beta_N\bigg[ \max_{k=1, \dots N}  \Psi_N(\tfrac{2\pi(k-1)}{N})   \ge  (1 +\delta)  \sqrt{\tfrac{2}{\beta}}\log N - \pi\bigg]  
 \]
and use a union bound in order to deduce the upper--bound for  $\max_\T \Psi_N$. 
\hfill $\blacksquare$\end{remark}

\subsection{Optimal rigidity: Proof of Corollary~\ref{cor:rig}.} \label{sect:rig}
%\label{oprig}

This is a direct consequence of Proposition~\ref{prop:TP2}, we give the details for completeness.  
Let us define the (centered) eigenvalue counting function
\begin{equation} \label{h}
h_N(\theta) =  \sum_{j\le N} \1_{\theta_j \in [0, \theta]}  - \frac{N \theta}{2\pi}   . 
\end{equation}
Since $\psi(\theta) = \frac{\theta-\pi}{2} $ for all $\theta \in (0, 2\pi)$, by formula \eqref{Psi}, 
the  function $\Psi_N$ is piecewise linear  on $\T \setminus \{ \theta_k\}_{k=1}^N $ and it jumps by $-\pi$ at the points $\theta_1, \dots, \theta_N$. 
Then, since $h_N(0) =0$, we have for all $\vartheta \in \T \setminus \{ \theta_k\}_{k=1}^N$,
\begin{equation} \label{eigcount}
h_N(\theta) = \frac{\Psi_N(0) - \Psi_N(\vartheta)}{\pi}.  
\end{equation}

From the asymptotics \eqref{ZG1}, we can deduce that for any $\delta\in [0,1]$,  
\[
\E^\beta_N \big[ e^{ \delta \Psi_N(0)} \big] \le   C_\beta N^{\delta^2/2\beta} , 
\]
for a constant $C_\beta$ which only depends on $\beta>0$ -- see \eqref{exp_moment} in the appendix~\ref{app:max}. 
This estimate implies that  for any $\delta\in [0,1]$,  
\begin{equation} \label{UB6}
\P_N^\beta\left[ \Psi_N(0) \ge \tfrac{\delta}{\beta} \log N \right] \le N^{-\delta^2/\beta} \E^\beta_N \big[ e^{ \delta \Psi_N(0)} \big] \le   C_\beta N^{-\delta^2/2\beta} . 
\end{equation}
Then, we deduce from Proposition~\ref{TP2} and Remark~\ref{rk:symmetry} that as $N\to+\infty$, 
\begin{equation} \label{UB7}
\frac{\max_\T|\Psi_N(\theta)|}{\log N} \to \sqrt{\frac{2}{\beta}} . 
\end{equation}
By formula \eqref{eigcount}, combining \eqref{UB7} and the estimate \eqref{UB6}, we obtain as $N\to+\infty$, 
\[
\frac{\max_\T| h_N(\theta)|}{\log N} \to  \frac 1\pi \sqrt{\frac{2}{\beta}} . 
\]
Finally since $\displaystyle \max_\T| h_N(\theta)| \le \max_{k=1, \dots , N}| h_N(\theta_k)| +1$ where $\theta_1, \dots , \theta_N$ are the \Cbeta\ eigenvalues, this implies that  for any $\beta >0$ and $\delta>0$, 
\begin{equation} 
\lim_{N\to+\infty} \P_N^\beta\left[ \frac{1-\delta}{\pi} \sqrt{\frac{2}{\beta}} \log N \le \max_{k=1,\dots , N}  \big| h_N(\theta_k)  \big| \le  \frac{1+\delta}{\pi} \sqrt{\frac{2}{\beta}}  \frac{\log N}{N} \right] =1 . 
\end{equation}
Since, by \eqref{h},  $h_N(\theta_k) = \frac{N}{2\pi}\left( \frac{2\pi k}{N} -   \theta_k \right)$ for $k=1, \dots , N$, this completes the proof. 

\section{Large deviation estimates for the eigenvalue counting function} \label{sect:LD}

Recall that we denote by $h_N$  the (centered) eigenvalue counting function \eqref{h}.
Note that almost surely, $h_N$ is a c\`adl\`ag function on $\T$ such that $\|h_N\|_{\infty} \le  N$ and  $h_N' = \widetilde{\mu}_N$
 in the sense that for any  function $f \in \Co^1(\T)$, we have 
\begin{equation} \label{h'}
\int f'(\theta) h_N(\theta)  d\theta =  - \int f d \widetilde{\mu}_N . 
\end{equation}

In this section, by using the connection between the eigenvalue  counting function and the logarithm of the characteristic polynomial, see formula \eqref{eigcount} above, we investigate the probability that $h_N$ takes extreme values. We obtain the following large deviation estimates. 

\begin{proposition} \label{prop:LDP}
Let $w\in \Co^1(\T)$ such that $\| w'\|_{L^1(\T)} \le \eta$ (where $\eta$ may depend on $N\in\N$). There exists $N_\beta\in\N$ such that for all $N\ge N_\beta$,
\begin{equation} \label{LDP}
\P^\beta_{N, w}\Big[ \max_\T|h_N| \ge \tfrac{\sqrt{2}}{\beta} \eta \log N \Big] 
\le \sqrt{\eta \log N} N^{1- \eta^2/2\beta}  . 
\end{equation}
\end{proposition}

For the proof of Proposition~\ref{prop:LDP}, we need the following Lemma which is an easy consequence of a result from Su \cite{Su09}. 
For completeness, the proof of Lemma~\ref{lem:maxh} is given in the Appendix~\ref{app:max}.

\begin{lemma} \label{lem:maxh}
There exists $N_\beta\in\N$, such that for all $N\ge N_\beta$ and any $t>0$ $($possibly depending on $N)$,
\[ 
\P^\beta_{N}\left[ \max_\T|h_N| \ge t\right] \le  3N e^{- \frac{\beta t^2}{\log N}}.
\]
\end{lemma}

Let us observe that since $h_N(\theta_k) = \frac{N}{2\pi}\left( \frac{2\pi k}{N} -   \theta_k \right)$
for $k=1,\dots, N$, Lemma~\ref{lem:maxh} immediately implies the \emph{rigidity estimate} \eqref{rig}. 
We now turn to the proof of Proposition~\ref{prop:LDP}.

\begin{proof}[Proof of Proposition~\ref{prop:LDP}]
Observe that by \eqref{h'}, we have
\[
\left| \int w \widetilde{\mu}_N \right| \le  \| w'\|_{L^1(\T)}  \max_\T|h_N| . 
\]
By Lemma~\ref{lem:maxh}, this implies that if $N\ge N_\beta$, then for any $t>0$, 
\[ 
\P^\beta_{N}\Big[ \left| \int w \widetilde{\mu}_N \right| \ge t\Big] \le 3 N 
 e^{- \frac{\beta t^2}{\eta^2\log N}}.
\]
Using this Gaussian tail--bound, we can estimate the Laplace transform of the linear statistics $\int w d\widetilde{\mu}_N$, we obtain
\[ \begin{aligned}
 \E^\beta_N \big[ e^{\int w d\widetilde{\mu}_N}\big] & = 1+ \int_0^{+\infty} \P^\beta_{N}\Big[  \int w \widetilde{\mu}_N  \ge t\Big] e^t dt \\
 & \le 1+ 3N \int_0^{+\infty}  e^{t -  \frac{\beta t^2}{\eta^2\log N}} dt . 
\end{aligned}\]
By completing the square, we obtain 
\[  
 \E^\beta_N \big[ e^{\int w d\widetilde{\mu}_N}\big] 
  \le 1+ 3N  e^{\frac{ \eta^2 \log N}{4\beta}} \int_1^{+\infty}  e^{-
 \frac{\beta}{\eta^2\log N}(t- \frac{\eta^2 \log N}{2\beta})^2} dt 
 \]
 So if $N$ is sufficiently large (depending on $\eta>0$ and $\beta>0$), this shows that
 \begin{equation}
  \label{est0}
  \E^\beta_N \big[ e^{\int w d\widetilde{\mu}_N}\big] 
 \le  1+ 3 \eta N^{1+ \eta^2/4\beta} \sqrt{\pi \log N/\beta} . 
\end{equation}
On the other hand, by Jensen's inequality and the fact that $\widetilde{\mu}_N$ is centered, we have
\[
 \E^\beta_N \big[ e^{\int w d\widetilde{\mu}_N}\big] \ge e^{ \E^\beta_N[ \int w d\widetilde{\mu}_N] } =1 .
\]
Therefore, by Cauchy-Schwartz inequality, we have for any $t>0$, 
\[ \begin{aligned}
\P^\beta_{N, w}\left[ \max_\T|h_N| \ge t \right] 
&\le \E^\beta_N \big[ \1_{ \max_\T|h_N| \ge t } e^{\int w d\widetilde{\mu}_N}\big]  \\
& \le  \sqrt{\P^\beta_N \big[\max_\T|h_N| \ge t \big]\E^\beta_N \big[e^{2\int w d\widetilde{\mu}_N}\big]}\\
&\le  C (\eta^2 \log N /\beta)^{1/4} N^{1+ \eta^2/2\beta} \ e^{- \frac{\beta t^2}{2\log N}} . 
\end{aligned}\]
For the last step, we used Lemma~\ref{lem:maxh} and the estimate \eqref{est0} replacing $w$ by $2w$ (this only changes $\eta$ by $2\eta$). 
From our last estimate, we obtain \eqref{LDP} by taking $t = \frac{\sqrt{2}}{\beta} \eta \log N$. 
\end{proof}

Finally it remains to give a proof of Proposition~\ref{prop:rig}. 

\begin{proof}[Proof of Proposition~\ref{prop:rig}]
Fix  $n\in\N$ and $R>0$. Like  \eqref{h'}, since $h_N' = \widetilde{\mu}_N$, we have for any $f\in C^n(\T^n)$, 
\[ \begin{aligned}
 \int_{\T^n} f(x_1,\dots, x_n)  \widetilde{\mu}_N(dx_1)\cdots \widetilde{\mu}_N(dx_n)  
 =  (-1)^n \int_{\T^n}  \frac{\d}{\d x_1} \cdots \frac{\d}{\d x_n} f(x_1,\dots, x_n) h_N(x_1) \cdots h_N(dx_n) dx_1 \cdots dx_n .
\end{aligned}\]
This implies that for any $f\in\mathscr{F}_{n,R}$,
\[
\left|  \int_{\T^n} f(x_1,\dots, x_n)  \widetilde{\mu}_N(dx_1)\cdots \widetilde{\mu}_N(dx_n) \right| \le R \left( \max_\T | h_N |\right)^n ,
\]
so that for any $\epsilon>0$, 
\[ \begin{aligned}
&\P^\beta_{N,w}\left[ 
\sup_{f\in \mathscr{F}_{n,R}}\bigg| \int_{\T^n} f(x_1,\dots, x_n)  \widetilde{\mu}_N(dx_1)\cdots \widetilde{\mu}_N(dx_n)  \bigg| \ge  R(\tfrac{\sqrt{2}}{\beta} \eta \log N)^{n}  \right]   \\
& \qquad 
\le \P^\beta_{N,w} \Big[   \max_\T | h_N | \ge \tfrac{\sqrt{2}}{\beta} \eta \log N \Big] . 
\end{aligned}  \]
Hence, the claim follows directly from from  Proposition~\ref{prop:LDP}.\end{proof}

\section{Proof of Theorem~\ref{clt:sine}}  \label{sect:VV}

Throughout this section, we use the notation from  Theorem~\ref{thm:coupling}. 
Let $w\in \Co^{3+\alpha}_c(\R)$, $0<\epsilon<1/4$ and $\nu > 4$. 
Without loss of generality, we assume that $\operatorname{supp}(w) \subseteq [-\frac 12, \frac12]$ -- if not we can consider instead the test function
$w_R = w(\cdot R)$ and observe that $\| w\|_{H^{1/2}(\R)} = \| w_R\|_{H^{1/2}(\R)}$. 
We also define $N(\nu) = \lfloor \exp(\nu^{1/4}) \rfloor$ and $ L(\nu) = \frac{N(\nu)}{2\pi\nu} $. 
We will need the following simple consequence of eigenvalue rigidity.

\begin{lemma} \label{lem:trunc}
There exists a random integer $\mathrm{N}_\epsilon \in\N$ such that for all $N\ge \mathrm{N}_\epsilon$, 
\[
\sum_{k\in \Z} w(\theta_k L)  = \sum_{|k| \le \nu} w(\theta_k L)   ,
\qquad\text{and}\qquad
\sum_{k\in \Z} w(\lambda_k \nu^{-1})  = \sum_{|k| \le \nu} w(\lambda_k \nu^{-1})    . 
\]
\end{lemma}

\begin{proof}
Using the estimate \eqref{rig} with the Borel--Cantelli Lemma combined with  Theorem~\ref{thm:coupling}, we see that there exists a random integer $\mathrm{N}_\epsilon$
such that it holds for all $N\ge \mathrm{N}_\epsilon$, 
\begin{equation} \label{rigidity}
\begin{cases}
 \big| \frac{N\theta_k}{2\pi} - k \big| \le  (\log N)^{1+\epsilon} & \forall k\in\Z \\
 \big| \frac{N\theta_k}{2\pi} - \lambda_k \big| \le N^{-\epsilon}
 &\forall |k| \le N^{1/4-\epsilon} 
\end{cases} .
\end{equation}
In particular, if $N(\nu)\ge \mathrm{N}_\epsilon$, we also have
\[
\min\left\{ \theta_\nu L(\nu) , \nu^{-1} \lambda_{\nu}  \right\} \ge 1 -\nu^{-1} (\log N(\nu) )^{1+\epsilon} \ge 1- \nu^{-\frac{3-\epsilon}{4}} \ge 1/2 . 
\]
Similarly, we can show that $\max\left\{ \theta_{-\nu} L(\nu) , \nu^{-1} \lambda_{-\nu}  \right\} \le -1/2$. Since we assume that $\operatorname{supp}(w) \subseteq [-\frac 12, \frac12]$, this implies that $w(\theta_k L ) = w(\nu^{-1} \lambda_k) = 0$ for all $|k| \ge \nu$. 
This completes the proof.
\end{proof}

\begin{remark}
It follows from the estimate \eqref{rigidity} that for any $\epsilon>0$, 
\[
\lim_{M\to+\infty} \P\left[ \max_{|k| \le M } |\lambda_k- k| \le (\log M )^{1+\epsilon}\right] =1 . 
\]
\hfill $\blacksquare$\end{remark}

We are now ready to prove our CLT for the Sine$_\beta$ process.

\begin{proof}[Proof of Theorem~\ref{clt:sine}]
Let $\mathrm{N_\epsilon}$ be as in Lemma~\ref{lem:trunc}. 
Since $w\in \Co^1_c(\R)$, we have for all $N\ge \mathrm{N_\epsilon}$, 
\[ \begin{aligned}
\bigg| \sum_{k\in \Z} w(\lambda_k \nu^{-1}) -  \sum_{k\in [N]} w(\theta_k L) 
\bigg| 
& \le \sum_{|k|<\nu} \big| w(\lambda_k \nu^{-1}) -   w(\theta_k L) \big|
 \\
&\le 2 \nu\max_{|k| < \nu} \big| \tfrac{N\theta_k}{2\pi}\nu^{-1} - \lambda_k \nu^{-1} \big|\  \| w'\|_{L^\infty(\R)} .
\end{aligned}\]  
Since  $\nu \le \big(\log N(\nu) \big)^{4}$, using the notation \eqref{LS}, this implies that for all $N(\nu)\ge \mathrm{N_\epsilon}$, 
\[
\bigg| \sum_{k\in \Z} w(\lambda_k \nu^{-1}) -  \int w_L d \widetilde{\mu}_N + N \int_{-\pi}^\pi w_L(x) \frac{dx}{2\pi}
\bigg|  \le 2 \| w'\|_{L^\infty(\R)} \max_{|k| \le (\log N)^4} \big| \tfrac{N}{2\pi}\theta_k - \lambda_k \big| . 
\]
By \eqref{rigidity}, the RHS of the above inequality converges to 0 almost surely as $N\to+\infty$. Moreover  since we have
$\displaystyle  N \int_{-\pi}^\pi w_L(x) \frac{dx}{2\pi} = \nu \int_\R w(x) dx$
and $\displaystyle\int w_L d \widetilde{\mu}_N \to  \mathscr{N}\big(0,\tfrac{2}{\beta}\| w\|_{H^{1/2}(\R)}^2\big) $ weakly as $N\to+\infty$, by Slutsky's Lemma, we conclude that  
\[
 \sum_{k\in \Z} w(\lambda_k \nu^{-1}) -  \nu \int_\R w dx 
 \to \mathscr{N}\big(0,\tfrac{2}{\beta}\| w\|_{H^{1/2}(\R)}^2\big) , 
  \]
  as $\nu\to+\infty$ in distribution.
\end{proof}

\bigskip

\appendix

\noindent{\bf\Large Appendices}
\vspace{-.3cm}

\section{Proof of Lemma~\ref{lem:maxh}} \label{app:max}

We make use of the explicit formula  \eqref{ZG1}.
By \eqref{Psi}, we see that almost surely:  $|\Psi_N(\theta)|\le \pi N$ for all $\theta\in\T$. So both sides of formula \eqref{ZG1} are analytic in the strip 
$\big\{ \gamma\in \C : | \Im \gamma| \le 2 \big\}$.  This implies that for any $\gamma\in\R$ and $\theta\in\T$, 
\begin{align} \notag
\E^\beta_N \big[ e^{\gamma \Psi_N(\theta)} \big] &= \prod_{k=0}^{N-1} \left| \frac{\Gamma(1+ \frac{k\beta}{2})}{\Gamma(1+ \frac{k\beta+\i \gamma}{2})}\right|^2 \\
&\label{ZG2}
=  \prod_{k=0}^{N-1}  \prod_{\ell=1}^{+\infty} \left( 1+ \left( \frac{\gamma}{k \beta + 2\ell} 
\right)^2 \right)
\end{align}
where we used properties of the Gamma function and the infinite product is convergent, see \href{https://dlmf.nist.gov/5.8}{https://dlmf.nist.gov/5.8}. 
Moreover, observe that
\[
\log\left( \prod_{k=0}^{N-1}  \prod_{\ell=1}^{+\infty} \left( 1+ \left( \frac{\gamma}{k \beta + 2\ell} 
\right)^2 \right) \right)  \le  \sum_{k=0}^{N-1} \sum_{\ell=1}^{+\infty} \left( \frac{\gamma}{k \beta + 2\ell} \right)^2 , 
\]
and using that  $\sum_{\ell=1}^{+\infty} \frac{1}{(\alpha + 2\ell)^2} \le \frac{1}{2\alpha}$ for any $\alpha>0$, we obtain 
\[ \begin{aligned}
\log\left( \prod_{k=0}^{N-1}  \prod_{\ell=1}^{+\infty} \left( 1+ \left( \frac{\gamma}{k \beta + 2\ell} 
\right)^2 \right) \right)  &\le \frac{\gamma^2}{2\beta}  \sum_{k=1}^{N-1} \frac{1}{k} +\frac {\gamma^2 \pi^2}{24} \\
&\le \frac{\gamma^2}{2\beta}( \log N+1) +\frac {\gamma^2 \pi^2}{24} .
\end{aligned}\]
This estimate implies that there exists a universal constant  $c_\beta>0$ such that for any $\gamma\in\R$ and $\theta\in\T$, 
\begin{equation} \label{exp_moment}
\E^\beta_N \big[ e^{\gamma \Psi_N(\theta)} \big] \le  e^{ c_\beta \gamma^2} N^{\gamma^2/2\beta} . 
\end{equation}

By \eqref{eigcount}, we have
\[
\max_\T |h_N| \le \frac 1\pi \max_\T |\Psi_N|  +  \frac 1\pi |\Psi_N(0)|  ,
\]
so that  for any $t \ge 1$,
\[
\P^\beta_{N}\left[ \max_\T |h_N|  \ge t\right]  \le \P^\beta_{N}\left[ \max_\T |\Psi_N|  \ge  \tfrac{\pi (t+1)}{2}\right]  + \P^\beta_{N}\left[ |\Psi_N(0)|  \ge  \tfrac{\pi (t-1)}{2}\right]   .
\]
Then, using the estimate \eqref{UB8},  by a union bound we obtain
\[
\P^\beta_{N}\left[ \max_\T |h_N|  \ge t\right]  \le   \sum_{k=0,\dots, N}  \P^\beta_{N}\left[ |\Psi_N(\tfrac{2\pi k}{N})|  \ge  \tfrac{\pi (t-1)}{2}\right]  . 
\]
Using the estimate \eqref{exp_moment} and Markov's inequality, this implies  that  for any $t \ge 1$,
\[ \begin{aligned}
\P^\beta_{N}\left[ \max_\T |h_N|  \ge t\right]
& \le  2   e^{- \frac{\gamma \pi (t-1)}{2}} \sum_{k=0,\dots, N}   \E^\beta_N \big[ e^{\gamma \Psi_N(\frac{2\pi k}{N})} \big]  \\
&\le 2  e^{ c_\beta \gamma^2} (N+1) N^{\gamma^2/2\beta}  e^{- \frac{\gamma \pi (t-1)}{2}} .
\end{aligned}\]
If we optimize and choose $\gamma = \frac{\beta \pi (t-1)}{2 \log N + \beta c_\beta}$, we obtain 
\[
\P^\beta_{N}\left[ \max_\T h_N \ge t\right] \le 2  (N+1) e^{- \frac{\beta \pi^2}{4} \frac{(t-1)^2}{2\log N + \beta c_\beta}} . 
\]
Hence, we conclude that there exists $N_\beta\in\N$, such that for all $N\ge N_\beta$ and any $t>0$ $($possibly depending on $N)$,
\[ 
\P^\beta_{N}\left[ \max_\T|h_N| \ge t\right] \le 3 N e^{- \frac{\beta t^2}{\log N}} .
\]

\section{Proofs of Lemma~\ref{lem:phi1} and Lemma~\ref{lem:phi2}} \label{app:lem}

Recall that for $0\le r<1$, the functions $\phi_r$ and $\psi_r$ are given by \eqref{phi} and \eqref{psi} respectively. We also define for $\theta\in (0,2\pi)$,
\begin{equation}\label{phipsi}
\phi(\theta) = \log|1-e^{\i\theta}|  
\qquad\text{and}\qquad
\psi(\theta)= \Im\log(1-e^{\i\theta}) = \frac{\pi- \theta}{2} . 
\end{equation}

\begin{proof}[Proof of Lemma~\ref{lem:phi1}]
\underline{Estimates \eqref{est1}.}
First, by the maximum principle, $\|\psi_r\|_\infty \le \|\psi\|_\infty = \pi$.
Then, an explicit computation gives $\psi_r'(\theta) =- \frac{r \cos\theta -r^2}{(1-r)^2 + 2r (1-\cos\theta)}$, so that
\[ \begin{aligned}
\int_\T |\psi_r'(\theta)| d\theta & \le  2(1-r)  \int_0^\pi \frac{1}{(1-r)^2 + 2r (1-\cos\theta)} d\theta + \int_0^\pi  \frac{2r(1-\cos\theta)}{(1-r)^2 + 2r (1-\cos\theta)} d\theta . 
\end{aligned}\]
The second integral is obviously bounded by $\pi$ and we can estimate the first by 
\begin{equation} \label{est10}
 \int_0^\pi \frac{1}{(1-r)^2 + 2r (1-\cos\theta)} d\theta  \le 
\frac{\epsilon}{(1-r)^2} + \pi \int_{\epsilon}^{\pi} \frac{d\theta}{\theta^2} 
\le \frac{\epsilon}{(1-r)^2}  + \frac{\pi}{\epsilon} ,
\end{equation}
where we used that $1-\cos\theta \ge \frac{\theta^2}{2\pi}$ for all $\theta\in [-\pi, \pi]$. Choosing $\epsilon = \sqrt{\pi}(1-r)$, we obtain 
\[
\int_\T |\psi_r'(\theta)| d\theta \le 4\sqrt{\pi} +\pi
\]
which proves the claim.

\smallskip
\underline{Estimates \eqref{est2}.} We clearly have $\| \phi_r \|_\infty = \phi_r(0) = \log|1- r|^{-1}$. 
Moreover, an explicit computation gives $\phi_r'(\theta) =- \frac{r \sin\theta}{(1-r)^2 + 2r (1-\cos\theta)}$, so that by a change of variables:
\[ \begin{aligned}
\int_\T |\phi_r'(\theta)| d\theta &= 2r  \int_0^\pi \frac{ \sin\theta}{(1-r)^2 + 2r (1-\cos\theta)} d\theta\\
&= \int_0^{2r} \frac{du}{(1-r)^2 + u} = 2\log\bigg(\frac{\sqrt{1+r^2}}{1-r} \bigg) . 
\end{aligned}\]

\smallskip
\underline{Estimates \eqref{est3}.} 
Since $\psi_r' = \sum_{k\ge 1} r^k \cos(k \cdot)$ and $\phi_r' = - \sum_{k\ge 1} r^k \sin(k \cdot)$, we have
\[ \|\psi_r'\|_\infty , \|\phi'_r\|_\infty \le  \frac{1}{1-r} .\]

\smallskip
\underline{Estimates \eqref{est5}.} 
If $z= r e^{\i\theta}$, then we verify that
$\displaystyle \phi_r'' (\theta)  = - \Re \frac{z}{(1-z)^2}$
and $\displaystyle \psi_r'' (\theta)  = - \Im \frac{z}{(1-z)^2}$. 
This shows that  for any $\theta \in\T$, 
  $| \phi_r'' (\theta) | ,  | \psi_r'' (\theta) |  \le |1-z|^{-2}$ and  by \eqref{est10}, we obtain
\[
 \int_\T |\psi_r'' (\theta) | d\theta , \ \int_\T |\phi_r'' (\theta) | d\theta \le 2  \int_0^\pi \frac{ d\theta }{(1-r)^2 + 2r (1-\cos\theta)} \le \frac{4\sqrt{\pi}}{1-r} .
\]
This completes the proof.
\end{proof}

\begin{proof}[Proof of Lemma~\ref{lem:phi2}]
Going through the proof of Proposition~\ref{prop:R}, we have for any $0<\epsilon\le 1$, 
\begin{equation} \label{R7}
R_0(\phi_r) \le R_3(\phi_r)  + \frac{\pi^2}{2} \left( \epsilon^{-1}  \|\psi_r'\|_{L^1(\T)} + \frac{\epsilon}{3} \left(  \left\|  \mathrm{M}_\epsilon \psi_r''' \right\|_{L^1(\T)} +  \| \psi_r'\|_\infty  \right)   \right) ,
\end{equation}
where $R_3$ is given by \eqref{R3} and $ \displaystyle \mathrm{M}_\epsilon \psi_r'''(x)= \hspace{-.3cm} \sup_{|\zeta -x| \le \epsilon} \hspace{-.1cm}|\psi_r'''(\zeta)| $.

If $z= r e^{\i\theta}$, then we verify that  $\displaystyle  \psi_r''' (\theta)  = \Re \frac{z(1+z)}{(1-z)^3}$.  
So, if $\epsilon \le \frac{1}{2}$, we  easily verify  that
\[
\mathrm{M}_\epsilon  \psi_r''' (\theta) \le  
\begin{cases} 
\frac{2}{(1-r^3)}&\text{if }\theta\in [-2\epsilon, 2\epsilon] \\
\frac{2r}{|1-r  e^{\i(\theta-\epsilon)}|^3}  &\text{if }\theta\in [2\epsilon, \pi/2] \cup  [3\pi/2, 2(\pi-\epsilon)] \\
8 &\text{if }\theta\in [\pi/2, 3\pi/2]
\end{cases} .
\]
Since $|1-r  e^{\i(\theta-\epsilon)}| \ge \sqrt{r/\pi} (\theta-\epsilon)$ 
%$|1-r  e^{\i(\theta-\epsilon)}| \ge \sqrt{2r(1-\cos(\theta-\epsilon))} \ge \sqrt{r/\pi} (\theta-\epsilon)$
 if $\theta\in [2\epsilon, \pi/2]$, this implies that if $1/\sqrt{\pi} \le r<1$, then
\[ \begin{aligned}
\int_\T  \mathrm{M}_\epsilon  \psi_r''' (\theta) d\theta & \le 
\frac{4\epsilon}{(1-r)^3} 
+ 4\pi^2   \int_{2\epsilon}^{\pi/2} \frac{d\theta}{(\theta-\epsilon)^3}
+8\pi  \\
&= \frac{4\epsilon}{(1-r)^3} + \frac{2\pi^2}{\epsilon^2}+\pi .\\
\end{aligned}\]

By \eqref{R7} combined with the estimates \eqref{est1} and \eqref{est3}, this implies that there exists a universal constant $C>0$ such that 
\begin{equation} \label{R30}
R_0(\phi_r) \le R_3(\phi_r)  +  C\left( \epsilon^{-1} +
 \frac{\epsilon^2}{(1-r)^3} +\frac{\epsilon}{1-r}   \right) . 
\end{equation}
We also need an estimate for  $R_3$. For $k\ge 0$, let $\epsilon_k = \epsilon (\log \epsilon^{-1})^k$.  By \eqref{R3}, We have
\begin{equation} \label{R31} 
\begin{aligned}
R_3(\phi_r)  %&  \le \frac{\pi^3}{4} \iint_{\T^2}  \1_{|x_1-x_2| \ge \epsilon} \bigg| \frac{ \psi_r(x_1)-\psi_r(x_2)}{(x_1-x_2)^3}\bigg| \d x_1\d x_2 \\
&\le   \frac{\pi^3}{4}  \sum_{k=0}^{K_\epsilon}\iint_{\T^2}  \1_{\epsilon_k\le  |x_1-x_2| \le \epsilon_{k+1}}  \bigg| \frac{ \psi_r(x_1)-\psi_r(x_2)}{(x_1-x_2)^3}\bigg| d x_1d x_2  , 
\end{aligned}
\end{equation} 
where $K_\epsilon = \inf\{ k\ge 0 : \epsilon (\log \epsilon^{-1})^k \ge \pi\}$. 
A similar argument as above shows that if $ \displaystyle \mathrm{M}_\delta \psi_r'(x)= \hspace{-.3cm} \sup_{|\zeta -x| \le \delta} \hspace{-.1cm}|\psi_r'(\zeta)| $, then for any $0<\delta\le \frac 12$, 
\[ \begin{aligned}
\int_\T \mathrm{M}_{\delta} \psi_r' (x) d x  
\le \frac{8\delta}{1-r} + 2\pi \log\delta^{-1}  . 
\end{aligned}\]
By Taylor's Theorem, since the function $\psi_r$ is smooth, this implies that for any $k\ge 0$, 
\[ \begin{aligned}
\iint_{\T^2}  \1_{\epsilon_k\le  |x_1-x_2| \le \epsilon_{k+1}}  \bigg| \frac{ \psi_r(x_1)-\psi_r(x_2)}{(x_1-x_2)^3}\bigg| \d x_1\d x_2 
& \le \int_\T \mathrm{M}_{\epsilon_{k+1}}\psi_r'(x_1) \left( \int_\T \1_{|x_1-x_2| \ge \epsilon_{k}} \frac{dx_2}{|x_1-x_2|^2} \right)  dx_1 \\
&\le  16 \frac{\epsilon_{k+1} \epsilon_k^{-1}}{1-r} + 4\pi  \epsilon_k^{-1}\log \epsilon_{k+1} ^{-1} \\
&\le 16 \log\epsilon^{-1} \left( \frac{1}{1-r} + \frac 1 \epsilon \right) . 
\end{aligned}\]
 We easily check that  $K_\epsilon +1 \le \frac{3\log \epsilon^{-1}}{\log \log \epsilon^{-1}}$, so by \eqref{R31}, there exists a constant $C>0$ such that 
 \begin{equation} \label{R32} 
 R_3(\phi_r) \le  C\frac{( \log \epsilon^{-1})^2}{\log \log \epsilon^{-1}} \left( \frac{1}{1-r} + \frac 1 \epsilon \right) .
 \end{equation}
 Combining the estimates \eqref{R30}, \eqref{R32}  and taking $\epsilon=1-r$, we obtain
 \begin{equation*} 
R_0(\phi_r) \le  C\frac{( \log (1-r)^{-1})^2}{(1-r)\log \log (1-r)^{-1}} .
\end{equation*}
By exactly the same argument, we obtain a similar bound for $R_0(\psi_r)$ and this completes the proof.
 \end{proof}

%\bibliographystyle{abbrvnat}
%\bibliography{/Users/gaultierlambert/Documents/Tex/NewBib.bib}

\end{document}